\documentclass[AMA,STIX1COL]{WileyNJD-v2}

\articletype{Article Type}%

\received{xxx}
\revised{xxx}
\accepted{xxx}

\usepackage{color}

\newcommand{\bl}{\color{blue}}

\renewcommand{\bl}{}
\newcommand{\nc}{\newcommand}
\nc{\ha}{\frac{1}{2}}
\nc{\tha}{\frac{3}{2}}
\nc{\s}{\widetilde}
\nc{\dst}{\displaystyle}
\nc{\gm}{\gamma}
\nc{\ga}{\Gamma}
\nc{\ka}{\kappa}
\nc{\eps}{\varepsilon}
\nc{\vep}{\varepsilon}
\nc{\hi}{\varphi}
\nc{\vfi}{\varphi}
\nc{\oa}{\Omega}
\nc{\Om}{\Omega}
\nc{\om}{\omega}
\nc{\ov}{\overline}
\nc{\lon}{\longrightarrow}
\nc{\scr}{\scriptstyle}
\nc{\ex}{\exists}
\nc{\fo}{\forall}
\nc{\pa}{\partial}
\nc{\und}{\underline}
\nc{\ze}{\zeta}
\nc{\si}{\sigma}
\nc{\tri}{\triangle}
\nc{\al}{\alpha}
\nc{\bt}{\beta}
\nc{\ts}{\textstyle}
\nc{\lf}{\left}
\nc{\ri}{\right}
\nc{\lm}{\hat{P}_{\chi_\Delta}}
\nc{\dt}{\delta}
\nc{\de}{\delta}
\nc{\te}{\theta}
\nc{\tl}{\tilde}
\nc{\wt}{\widetilde}
\nc{\p}{\prime}
\nc{\m}{\mu}
\nc{\R}{\mathbb{R}}
\nc{\B}{\mathbb{B}}
\nc{\N}{\mathbb{N}}
\nc{\C}{\mathbb{C}}
\newcommand{\be}{\begin{equation}}
\newcommand{\ee}{\end{equation}}

\renewcommand{\P}{{\cal P}}

\newcommand{\V}{{\cal V}}

\newcommand{\W}{{\cal W}}

\allowdisplaybreaks

\begin{document}

\title{Singular Localised Boundary-Domain Integral Equations of Acoustic Scattering
by Inhomogeneous
Anisotropic Obstacle}
\author[1]{Otar Chkadua}
\author[2]{Sergey E. Mikhailov*}
\author[3,4]{David Natroshvili}

\authormark{CHKADUA \textsc{et al}}

\address[1]{\orgdiv{A.Razmadze Mathematical Institute},
\orgname{I.Javakhishvili Tbilisi State University}, \orgaddress{\state{Tbilisi},
\country{Georgia}}}

\address[2]{\orgdiv{Department of Mathematics}, \orgname{Brunel University London}, \orgaddress{
 \country{UK}}}

\address[3]{\orgdiv{Department of Mathematics}, \orgname{Georgian Technical University}, \orgaddress{\state{Tbilisi}, \country{Georgia}}}  

\address[4]{\orgdiv{I.Vekua Institute of Applied Mathematics}, \orgname{I.Javakhishvili Tbilisi State University}, \orgaddress{\state{Tbilisi}, \country{Georgia}}}

\corres{*S.E. Mikhailov, Department of Mathematics, Brunel University London,
Uxbridge, UB8 3PH,  UK. \email{sergey.mikhailov@brunel.ac.uk}}


\abstract[Summary]
{
 We consider the time-harmonic  acoustic wave scattering
by a bounded {\it anisotropic inhomogeneity} embedded in an unbounded
{\it anisotropic}  homogeneous medium. The material parameters   may have
 discontinuities across the interface between the inhomogeneous interior and homogeneous exterior regions.
The corresponding mathematical problem is formulated as a transmission problems for  a second
order elliptic partial differential equation  of Helmholtz type with discontinuous variable coefficients.
Using a localised quasi-parametrix based  on the harmonic fundamental solution, the transmission
problem  for arbitrary values of the frequency parameter is reduced equivalently to a  system of
{\it singular localised boundary-domain integral equations}.
Fredholm properties of the corresponding {\it localised boundary-domain integral operator} are studied and its
invertibility is established in appropriate Sobolev-Slobodetskii and Bessel potential spaces,
which implies existence and uniqueness results for
the localised boundary-domain integral equations  system  and
the corresponding acoustic scattering transmission  problem.
}


\maketitle
\section{Introduction}
We consider the time-harmonic  acoustic wave scattering
by a bounded {\it anisotropic inhomogeneous obstacle} embedded in an unbounded
{\it anisotropic}  homogeneous medium.
 We assume that the material parameters and  speed of sound  are functions of position
within the  inhomogeneous  bounded obstacle.
The physical model problem with a frequency parameter $\omega\in \mathbb{R}$ is formulated mathematically
as a  transmission problem  for a second
order elliptic partial differential equation  with variable coefficients
$A_2(x,\partial_x)\,u(x)\equiv \partial_{x_k}\,\big(a^{(2)}_{kj}(x)\,
\partial_ {x_j} u(x)\big) +\omega^2\,\kappa_2(x)\,u(x)=f_2$
in the {\it inhomogeneous anisotropic}  bounded region  $\Omega^+\subset \mathbb{R}^3$
and for a Helmholtz type equation with constant coefficients
$A_1(\partial_x)u(x)\equiv\, a^{(1)}_{kj}\, \partial_{x_k}\partial_ {x_j} u(x)  +\omega^2\kappa_1\,u(x)=f_1$
in the {\it homogeneous anisotropic} unbounded region  $\Omega^-=\mathbb{R}^3\setminus \overline{\Omega^+}$.
The material parameters $a^{(q)} _{kj}$  and  $\kappa_q$ are not assumed
to be  continuous across the interface  $S=\pa \Omega^-=\pa \Omega^+$ between the inhomogeneous
interior and homogeneous  exterior regions.
  The transmission conditions are assumed on the interface,  relating the interior and
  exterior traces of the wave amplitude $u$ and its  co-normal derivative  on $S$.

{\bl The transmission problems for the  Helmholtz equation, i.e., when $A_2(x,\pa)=A_1(\pa)=\Delta +\omega^2$, which corresponds to a {\it homogeneous isotropic media}, are well studied in the case of smooth and Lipschitz interface (see 
Costabel \& Stephan\cite{CS}, Kleinman \& Martin\cite{KM}, Kress \& Roach\cite{KR}, Torres \& Welland\cite{TW} and the references therein).}

The  {\it special isotropic transmission problems}, when  $A_2(x,\partial_x)=\Delta +\omega^2\kappa_2(x)$
and $A_1(\partial_x)=\Delta +\omega^2$ is the Helmholtz operator
are also well   presented  in the literature
(see Colton \& Kress\cite{CK2}, N{\'e}d{\'e}lec \cite{Ned}, and the references therein).
  The  acoustic scattering problem in the whole space  corresponding to   a more general  isotropic case,
when $ a^{(2)} _{kj}(x)= a (x)\,\delta_{kj}$, where $\delta_{kj}$ is Kronecker's delta,  and $A_1(\partial_x)=\Delta +\omega^2$,   was analysed by the indirect boundary-domain integral equation method   by Werner
 \cite{Wer1}-\cite{Wer2}.
Applying the potential method based on the Helmholtz fundamental solution, P.Werner    reduced
the problem to the {\it Fredholm-Riesz type
integral equations} system and   proved  its unique solvability.  The same problem by the direct
method  was  considered by Martin \cite{Mart}, where the problem  was  reduced to a singular
integro-differential equation in the inhomogeneous bounded region $\Omega^+$. Using
the uniqueness and existence results obtained in by  Werner\cite{Wer1}{-}\cite{Wer2},
the equivalence of the integro-differential equation  to the
initial transmission problem and its unique solvability were shown for special
type right-hand side functions associated with Green's third formula.

Note that the wave scattering problems for the general inhomogeneous
anisotropic case described above can be studied by the variational method incorporated with
the non-local approach
and also by the classical potential method when  the corresponding fundamental
solution is available in an  explicit form.
  However,   fundamental solutions  for second
order elliptic partial differential equations with variable coefficients are not
available in explicit form, in general.
Application of the potential method based on the corresponding Levi function,
which always can be constructed explicitly, leads to Fredholm-Riesz type integral equations but
invertibility of the corresponding integral operators can be proved only for particular cases (see Miranda\cite{Mir}).

Our goal here is to show that the acoustic transmission problems for anisotropic  heterogeneous structures can
be equivalently reformulated as systems of singular {\it localized boundary-domain integral
equations} (LBDIEs)  with the help of a {\it localized harmonic paramerix}
based on the harmonic  fundamental solution, which is a {\it quasi-parametrix}
for the considered PDEs of acoustics,  and to prove that the corresponding
singular {\it localized boundary-domain integral operators} (LBDIO) are invertible
 for an arbitrary value of the frequency parameter.
Beside a pure mathematical interest, these results seem to be
important from the point of view of applications, since LBDIE system can be applied
in constructing  convenient numerical algorithms (cf. Mikhailov \cite{Mik1}, Zhu et al\cite{ZZA1, ZZA2}).
{{The main novelty} of the paper is in application of the singular localized
boundary-domain integral equations method to the problem of acoustic transmission
 through a penetrable, anisotropic, inhomogeneous obstacle.}


The paper is organized as follows. First, after mathematical formulation of the problem,
we  introduce layer and volume  potentials based on a localized harmonic parametrix and derive basic
integral relations in bounded inhomogeneous and unbounded homogeneous anisotropic regions. Then we reduce the transmission problem  under consideration to the localized boundary-domain singular  integral equations   system and
prove  the equivalence theorem for arbitrary values of the frequency parameter,
 which plays a crucial role in our analysis.
Afterwards, applying the Vishik-Eskin approach,
we investigate Fredholm properties of the corresponding matrix LBDIO, containing singular integral operators
over  the interface surface and the bounded region occupied by the inhomogeneous obstacle,
  and prove  invertibility of the LBDIO in appropriate Sobolev-Slobodetskii and Bessel potential spaces.
This invertibility property implies then, in particular, existence and uniqueness results for
the LBDIE system and the corresponding original transmission problem.

 Next, we analyze also an alternative non-local approach based on coupling of variational
and boundary integral equation methods, which reduces the transmission problem for unbounded
composite structure to the variational equation containing a coercive sesquilinear form which
lives on the bounded inhomogeneous region and the interface manifold.
 Both approaches presented in the paper can be applied in the study of similar wave scattering problems for
multi-layer piecewise inhomogeneous anisotropic structures.

Finally, for the readers convenience,  we collected necessary auxiliary material related to classes
of localizing functions, properties of localized potentials and anisotropic radiating potentials
in three brief appendices.



\section{Formulation of the   transmission  problem}
\label{ss2.1}

Let $\Omega^+=\Omega_2$  be a bounded domain  in $\mathbb{R}^3$ with a simply connected boundary  $\partial\Omega_2=S$,
and $\Omega^-=\Omega_1:={\mathbb{R}^3}\setminus \overline{\Omega}_2$.
 For simplicity, we assume that $S\in C^\infty$ if not otherwise stated.
 Throughout the paper $n=(n_1,n_2,n_3)$ denotes the  unit normal vector to $S$  directed
outward  the  domain $\Omega_2$.

We assume that the propagation region of a time harmonic acoustic wave $u^{tot}$ is the whole space  $\mathbb{R}^3$ which consists of an inhomogeneous part $\Omega_2$ and  a   homogeneous  part $\Omega_1$.
Acoustic wave propagation  is governed  by the  uniformly elliptic
 second order scalar partial differential equation
\begin{eqnarray}
\label{2.1-w}
A\,u^{tot}(x)\equiv \partial_{k}\,\big(a_{kj}(x)\, \partial_ {j} u^{tot}(x)\big) +\omega^2\,\kappa(x)\,u^{tot}(x)= f(x),\;\;\;x\in \Omega_2\cup\Omega_1,
\end{eqnarray}
where  $\pa_x=(\pa_1,\pa_2,\pa_3)$, $\pa_j=\pa_{x_j}=\pa/\pa x_j$,
$a_{kj}(x)=a_{jk}(x)$ and $\kappa(x)$ are real-valued functions,
 $\omega\in \mathbb{R}$ is a frequency parameter, while
   $f\in L_{2,comp}(\mathbb{R}^3)$   is the volume force amplitude. Here and in what follows, the Einstein summation  by repeated indices from $1$ to $3$ is assumed.

 Note that in the mathematical model of an inhomogeneous absorbing medium the  function $\kappa$  is complex-valued,
 with nonzero real and imaginary parts, in general
(see, e.g., Colton \& Kress\cite{CK2}, Ch. 8). Here we treat only the case when the  $\kappa$  is a real-valued function
but it should be mentioned that the complex-valued case can be also considered by the approach developed here.

In our further analysis, it is assumed that the real-valued variable coefficients $a_{kj}$ and $ \kappa$ are
constant  in the homogeneous unbounded region
$\Omega_1$ and the following relations hold:
\begin{align}
\begin{array}{l}
\label{2.2-w}
a_{kj}(x)=a_{jk}(x)=\left\{
\begin{array}{lll}
a^{(1)}_{kj} & \text{for}  & x\in {\Omega_1},\\
a^{(2)}_{kj}(x) & \text{for} & x\in {\Omega_2},
\end{array}
\right.
\quad
\kappa(x)=\left\{
\begin{array}{lll}
 \kappa_1>0 & \text{for}& x\in {\Omega_1},\\
\kappa_2(x)>0 & \text{for}  & x\in {\Omega_2},
\end{array}
\right.
\end{array}
\end{align}
where $a^{(1)}_{kj}$ and $\kappa_1$ are constants, while $a^{(2)}_{kj}$ and $\kappa_2$ are smooth function in $\overline{\Omega}_2$,
\begin{align}
\label{2.3-w}
a^{(2)}_{kj}, \;\kappa_2\in C^{\,2}(\overline{\Omega}_2),\;\;\;\;j,k=1,2,3.
\end{align}
Moreover, the matrices $ \mathbf{a}_q  =\big[\,a^{(q)}_{kj}\,\big ]_{k,j=1}^3$
are uniformly positive definite, i.e., there are positive constants $c_1$ and $c_2$   such that
\begin{eqnarray}
\label{2.4-w}
c_1\,|\xi|^2\leq  a^{(q)}_{kj} (x)\, \xi_k\,\xi_j\leq c_2\,|\xi|^2
\qquad \forall\;  x\in \overline{\Omega}_q, \;\; \forall\;\xi\in \mathbb{R}^3, \;\;q=1,2.
\end{eqnarray}
We do not assume that the coefficients $a_{kj}$ and $ \kappa$ are continuous across $S$ in general, i.e., the case
$a^{(2)}_{kj}(x)\neq a^{(1)}_{kj}$ and $\kappa_2(x)\neq   \kappa_1$ for $x\in S$ is covered by our analysis.
Further, let us denote
\begin{align}
\label{2.5-w}
&
A_1 v(x):=  a^{(1)}_{kj}\, \partial_{x_k}\partial_ {x_j} v(x)  +\omega^2  \kappa_1\,v(x) \;\;
\text{for} \;\; x\in \Omega_1 ,\\
\nonumber
&
A_2 v(x):= \partial_{x_k}\,\big(a^{(2)}_{kj}(x)\, \partial_ {x_j} v(x)\big) +\omega^2\,\kappa_2(x)\,v(x) \;\;
  \text{for} \;\;  x\in \Omega_2 .
\end{align}
  We will often write $A_1$ instead of $A_1(\partial_x)$ and $A_2$ instead of $A_2(x,\partial_x)$, when this does not lead to a confusion.
\\
 For a function $v$ sufficiently smooth in $\Omega_1$ and $\Omega_2$, the {\it classical} co-normal derivative operators,  $T_{cq}^\pm $ are well defined as
\begin{align}
\label{2.7-w}
&
   T_{cq}^\pm  v(x) :=  a^{(q)}_{kj}\, n_k(x)\,\gamma^\pm(\partial_ {x_j} v(x)) , \;\;  x\in S,\;\;q=1,2;
\end{align}
here the symbols $\gamma^{+}$ and $\gamma^{-}$ denote one-sided   boundary  trace operators
on $S$ from the interior and exterior domains respectively.
Their continuous right inverse operators, which are non-uniquely defined, are denoted by symbols $(\gamma^\pm)^{-1}$.

By $H^s(\Omega)=H^s_2(\Omega)$, $H^s_{loc}(\Omega)=H^s_{2,\,loc}(\Omega)$,
$H^s_{comp}(\Omega)=H^s_{2,\,comp}(\Omega)$ and $H^s(S)=H^s_2(S)$, $s\in \R$,  we denote the  $L_2$-based
Bessel potential spaces on an open domain $\Omega\subset \mathbb{R}^3$ and on a closed manifold $S$ without boundary,
while $\mathcal{D}(\Omega)$ stands for the space of infinitely differentiable test functions with support in $\Omega$.
Recall that $H^0(\Omega)=L_2(\Omega)$ is a space of square integrable functions in $\Omega$.
  Let the symbol $r_{\Omega}$ denote the restriction operator onto $\Omega$.

 Since the boundary traces of gradients, $\gamma^\pm(\partial_ {x_j} v(x))$ are generally
 not well defined on   functions from $H^1(\Omega_q)$, the classical co-normal derivatives \eqref{2.7-w}
  are not well defined on such functions either, cf. Mikhailov \cite{MikArxiv2015}, Appendix A,
   where an example of such function, for which the classical co-normal derivative exists at no boundary point.
 Let us introduce the following subspaces of $H^1({\Omega_2})$
and $H^1_{loc}({\Omega_1})$
to which the classical co-normal derivatives can be continuously extended, cf., e.g., Grisvard \cite{Grisvard1985}, Costabel \cite{Costabel1988}, Mikhailov \cite{Mik3}:
\begin{align*}
H^{1,\,0}({\Omega_2}; A_2 ):=\{\,v\in H^{1}(\Omega_2)\,:\, A_2v\in H^{0}({\Omega_2} )\,\},
\qquad
H^{1,\,0}_{loc}({\Omega_1}; A_1 ):=\{\,v\in H^{1}_{loc}(\Omega_1)\,:\, A_1v\in H^{0}_{  loc}({\Omega_1} )\,\}\,.
\end{align*}
  We will also use the corresponding spaces with the Laplace operator $\Delta$ instead of $A_q$.

Motivated by  the first Green identity well known for smooth functions,
the  classical  co-normal derivative operators  \eqref{2.7-w} can be extended by
continuity to functions from the spaces $H^{1,\,0}_{loc}(\Omega_1; A_1)$ and $H^{1,\,0}(\Omega_2; A_2)$
giving the  {\it canonical} co-normal derivative operators, $T^\pm_1$ and $T^+_2$, defined in the weak form as
\begin{align}
\label{2.14-w}
 \langle T^+_{q} u \,,\,g\rangle_{S } :=&
\int_{\Omega_2 }
 [ a^{({q})}_{kj}(x)\,\pa_{j}u (x)\; \pa_{k}(\gamma^+)^{-1}g(x)-
\omega^2\kappa_{q}(x) u (x)\,(\gamma^+)^{-1}g(x)]\,dx \nonumber \\
  &+\int_{\Omega_{q} }A_{q} u (x)\;(\gamma^+)^{-1}g (x)\,dx,\qquad u\in H^{1,\,0}(\Omega_2; A_{q}),
  \quad\forall\ g\in H^{\frac{1}{2}}(S),\\
\label{2.15-w1}
 \langle T^-_1 u \,,\,g \rangle_{S } :=&
-\int_{\Omega_1 }[  a^{(1)}_{kj} \,\pa_{j}u (x)\; \pa_{k}(\gamma^-)^{-1}g (x)-\omega^2\kappa_1 u (x)\,  (\gamma^-)^{-1}g (x)]\,dx\nonumber\\
  &-\int_{\Omega_1 }A_1u(x)\,(\gamma^-)^{-1}g (x)\,dx, \qquad u\in H^{1,\,0}_{loc}(\Omega_1; A_1),
  \quad\forall\ g\in H^{\frac{1}{2}}(S),
  \end{align}
  where
$(\gamma^+)^{-1}:H^{\frac{1}{2}}(S)\to H^1(\Omega_2 )$ and
  $(\gamma^-)^{-1}:H^{\frac{1}{2}}(S)\to H^1_{comp}(\Omega_1 )$
 are the right inverse operators to the trace operators $\gamma^\pm$, and
  the   angular brackets $\langle\cdot\,,\,\cdot\rangle_S$ should be understood as duality pairing of
  $H^{-\frac{1}{2}}(S)$ with $H^{\frac{1}{2}}(S)$  which extends the usual bilinear $L_2(S)$ inner product.

The canonical co-normal derivatives  $T^-_2 u$ and $ T^+_1 u$ can be defined analogously for functions from the spaces
$  H^{1,\,0}_{loc}(\Omega_1; A_2)$ and $ H^{1,\,0} (\Omega_2; A_1)$, respectively, provided that the variable coefficients
$a^{(2)}_{kj}(x)$ and $\kappa_2(x)$ are continuously extended from $\Omega_2$ to the whole space $\mathbb{R}^3$ preserving the
smoothness. It is evident that for functions from the space
$H^2(\Omega_2)$ and $H^2_{loc}(\Omega_1)$ the classical and canonical co-normal derivative operators coincide.
Concerning the {\it canonical} and {\it generalized}  co-normal derivatives in  wider functional spaces,   see
Mikhailov \cite{Mik3}.

 For two times continuously differentiable function $w$ in a neighbourhood of $S$,  we employ also the notation
$T_q(x,\pa_x) w:= a^{(q)}_{kj}\, n_k(x)\,(\partial_ {x_j}w(x))$, $x\in S$, to denote the restriction of
$T_q(x,\pa_x) w$ to $S$, which coincides with both the classical and the canonical co-normal derivatives.

Recall that, the definitions of the co-normal derivatives $T^\pm_q$ do not depend on the
choice of  the right  inverse operators $(\gamma^\pm)^{-1}$
  and the following Green's first and second identities hold (cf. Mikhailov \cite{Mik3}, Theorem 3.9),
\begin{align}
\label{2.14-wG}
 & \big\langle T^+_qu \,,\, {\gamma^+ v}\big\rangle_{S} =
\int_{\Omega_2 }
 \big[ a^{(q)}_{kj} \,\pa_{j}u  \; {\pa_{k}v } -
\omega^2\kappa_q  u  \, {v }\,\big]\,dx
 +\int_{\Omega_2 } {v }A_q u  \,dx,\;\;
 \ u\in H^{1,\,0}(\Omega_2; A_q),\ v\in H^{1}(\Omega_2),\;\;q=1,2,
  \\
  \nonumber
 &   \big\langle T^+_2 u \,,\, {\gamma^+ v}\big\rangle_{S}-
\big\langle  { T^+_2 v} \,,\, \gamma^+ u \big\rangle_{S} =
  \int_{\Omega_2 }\big[\, {v }\,A_2 u  -u \, A_2  {v }\,\big]dx,\quad   u,v\in H^{1,\,0}(\Omega_2; A_2),
   \\
\label{2.14-wG1}
 &\big\langle T^-_1 u \,,\, {\gamma^- v}\big\rangle_{S}=
-\int_{\Omega_1 }
 \big[\, a^{(1)}_{kj}\,\pa_{j}u  \,  {\pa_{k}v }-
\omega^2\kappa_1 u  \, {v }\,\big]\,dx
 -\int_{\Omega_1} {v }\,A_1u  \,dx,\quad
 \ u\in H^{1,\,0}_{loc}(\Omega_1; A_1),\ v\in H^{1}_{comp}(\Omega_1).
 \end{align}

By  $Z({\Omega_1})$ we denote  a sub-class  of complex-valued functions from $H^{1}_{loc}(\Omega _1)$
satisfying the Sommerfeld radiation conditions at infinity (see Vekua \cite{Vek},
Colton \& Kress \cite{CK2}  for the Helmholtz operator and
 Vainberg \cite{Vain}, Jentsch et al \cite{JN2} for the ``anisotropic'' operator
$A_1$ defined by \eqref{2.5-w}).
 Denote by $S_{\om}$ the characteristic surface (ellipsoid) associated with the operator
$A_1$,
 \[
 a^{(1)}_{kj}\xi_k\,\xi_j- \om^2 \kappa_1 =0,\;\;\;\xi\in\R^3.
\]
For an arbitrary vector $\eta\in\R^3$ with  $|\eta|=1$
there exists only one point $\xi(\eta)\in S_{\om}$ such that the outward unit
normal vector $n(\xi(\eta))$ to $S_{\om}$  at the point  $\xi(\eta)$ has the
same direction as $\eta$, i.e., $n(\xi(\eta))=\eta$.
Note that  $\xi(-\eta)=-\xi(\eta)\in S_{\om}$  and $n(-\xi(\eta))=-\eta. $
It can easily be verified that
\be
\label{2.11-w}
\xi(\eta)=\om\,  \kappa_1 ^{1/2} \,( {\bf a}_1^{-1}  \eta\cdot \eta)^{-1/2} \,\,
 {\bf a}^{-1}_1  \eta,
\ee
where   ${\bf a}_1^{-1}$ is the matrix inverse to $ {\bf a}_1:=  \big[\,a^{(1)}_{kj}\,\big ]_{k,j=1}^3$.

\begin{definition}
A  complex-valued function $v$ belongs to the class $Z(\Omega_1)$
if there   exists a ball $B(R)$ of radius $R$ centered at the origin  such that
$v\in C^1(\Omega_1\setminus B(R))$, and
$v$ satisfies the Sommerfeld radiation conditions associated with the operator $A_1(\partial)$
for sufficiently large $|x|$,
\be
\label{2.12-w}
v(x)=\mathcal{O}(|x|^{-1}),\quad
\pa_k v(x) -  i \xi_k(\eta)v(x)=\mathcal{O}(|x|^{-2}), \quad k=1,2,3,
\ee
where  $\xi(\eta)\in S_{\omega}$ corresponds to the
vector $\eta=x/|x|$ (i.e., $\xi(\eta)$ is given by \eqref{2.11-w} with $\eta=x/|x|$).
\end{definition}
Notice that due to the ellipticity of the operator $A_1(\pa_x)$, any solution to the constant coefficient
homogeneous equation $A_1(\pa_x)v(x)=0$ in an open region $\Omega\subset \mathbb{R}^3$  is a real analytic function
of $x$ in $\Omega$.

Conditions \eqref{2.12-w} are equivalent to the classical
Sommerfeld radiation conditions for the Helmholtz equation if $A_1(\partial)=\Delta(\pa)+\omega^2$,
i.e., if  $\kappa_1=1$ and  $a^{(1)}_{kj} =\delta_{kj}$, where $\delta_{kj}$ is the Kronecker delta.
 There  holds the following analogue of the classical Rellich-Vekua lemma
(for details see Jentsch et al \cite{JN2}, Natroshvili et al \cite{NKT}).
\begin{lemma}
\label{L2.1}
Let $  v \in Z({\Omega_1})$ be a solution of the equation $A_1(\partial_x) v =0$ in ${\Omega_1}$
and let
\begin{align}
\label{2.13-w}
\lim_{R\to+\infty} \mbox{\rm Im}\;\Big\{
\int_{\Sigma_R}\ov{v(x)}\;T_1(x,\pa_x)v(x)\,d\Sigma_R\Big\}=0,
\end{align}
where $\Sigma_R$ is the sphere  with radius $R$ centered at the origin.
Then $  v =0$ in ${\Omega_1}$.
\end{lemma}
\begin{remark}
\label{r2.2-1}
For $x\in \Sigma_R$ and $\eta=x/|x|$ we have $n(x)=\eta$ and in view of  \eqref{2.7-w} and \eqref{2.12-w} for a
function  $v\in Z({\Omega_1})$   we get
$$
 T_1(x,\pa_x)  v (x)= a^{(1)}_{kj}n_k(x)\,[\,i\,\xi_j(\eta)\,  v (x)]+\mathcal{O}(|x|^{-2})=
i\,a^{(1)}_{kj}\eta_k\, \xi_j(\eta)\, v (x)+\mathcal{O}(|x|^{-2})\,.
$$
Therefore, by \eqref{2.11-w}  and  the symmetry condition   $a_{kj}=a_{jk}$,   we arrive at the relation
\[
\overline{  v (x)}\, T_1 (x,\pa)  v (x)= i \,\omega  \kappa_1^{1/2}\,| v (x)|^2\,( {\bf a}_1^{-1}  \eta\cdot
\eta)^{-1/2} \,{\bf a} _1 \eta
\cdot  {\bf a}^{-1} \eta +\mathcal{O}(|x|^{-3})
=i\,\omega  \kappa_1^{1/2}\,( {\bf a}^{-1}_1  \eta\cdot \eta)^{-1/2}\,|  v (x)|^2 +\mathcal{O}\big(|x|^{-3}\big),
\]
On the other hand,  matrix ${\bf a}_1$ is  positive definite, cf. \eqref{2.4-w}, which implies positive definiteness of the inverse matrix ${ \bf a}^{-1}_1$.
Hence there are positive  constants $\delta_0$ and $\delta_1$ such that the inequality
$0<\delta_0\leqslant  ({ \bf a}^{-1}_1  \eta\cdot \eta)^{-\frac{1}{2}}\leqslant  \delta_1< \infty$ holds  for all $\eta \in \Sigma_1$.
 Consequently, \eqref{2.13-w} for $\omega\neq 0$ is equivalent to the condition in the well
known Rellich-Vekua lemma in the theory of the Helmholtz equation,
Vekua \cite{Vek}, Rellich \cite{Rell},  Colton \& Kress \cite{CK2},
\[
\lim_{R\to+\infty}
\int_{\Sigma_R} |v(x)|^2\,d\Sigma_R = 0.
\]
\end{remark}

In the unbounded region ${\Omega_1}$, we have a total wave field $u^{  tot}=u^{  ins}+u^{\rm sc}$, where
 $u^{  inc}$ is a wave motion initiating  known incident field and $u^{ sc}$ is a radiating unknown
scattered field.   It is often  assumed that the incident field is defined in the whole of ${\mathbb{R}}^3$,
being  for example
a corresponding plane wave which solves the homogeneous equation $A_1 u^{inc}=0$ in ${\mathbb{R}}^3$
but does not satisfy the Sommerfeld radiation conditions at infinity.
Motivated by relations \eqref{2.2-w}, let us set $u_1(x):=u^{sc}(x)$ for $x\in \Omega_1$ and
 $u_2(x):=u^{tot}(x)$ for $x\in \Omega_2$.

Now we formulate the transmission  problem associated with the time-harmonic  acoustic wave scattering by a bounded  anisotropic inhomogeneity  embedded in an unbounded
 anisotropic   homogeneous medium:

  \textit{Find complex-valued  functions $u_1\in H^{1,\,0}_{loc}({\Omega_1},A_1)\cap Z(\Omega_1)$ and
  $u_2\in   H^{1,\,0}({\Omega_2},A_2)$
 satisfying the differential equations
\begin{align}
\label{2.16-w}
&
  A_1u_1(x)=f_1(x) \;\;
\text{for} \;\; x\in \Omega_1,\\
\label{2.17-w}
&
 A_2u_2(x)=f_2(x) \;\;
  \text{for} \;\;  x\in \Omega_2 ,
\end{align}
and the transmission conditions on the interface $S$,
\begin{align}
\label{2.18-w}
&
\gamma^+u_2 -\gamma^-u_1=\varphi_{_0}      \;\;\text{on} \;\; S,
\\
\label{2.19-w}
&
 T^+_2 u_2 -   T^-_1  u_1 =\psi_{_0}      \;\;\text{on} \;\; S,
\end{align}
where
\begin{align}
\label{2.20-w1}
\begin{array}{c}
f_2:=r_{ \Omega_2}  f\in H^0(\Omega_2), \;\;\;
f_1:=r_{ \Omega_1} f\in H^0_{comp}({\Omega_1}),\;\;\;f\in H^0_{comp}(\mathbb{R}^3),
\;\;\;
 \varphi_{_0}\in H^{\frac{1}{2}}(S),  \;\;
\psi_{_0}\in H^{-\frac{1}{2}}(S).
\end{array}
\end{align}
}
In the above setting, equations \eqref{2.16-w} and \eqref{2.17-w} are understood in the distributional sense,
the Dirichlet type transmission condition \eqref{2.18-w} is understood in the usual trace sense, while
  the Neumann type transmission condition   \eqref{2.18-w} is understood  in the canonical co-normal derivative
   sense defined by the relations \eqref{2.14-w}-\eqref{2.15-w1}.\\
  If the interface continuity of $u^{tot}$ and its co-normal derivatives is assumed, then $\varphi_{_0}=  \gamma^- u^{inc} $, $\psi_{_0}= T^-_1 u^{inc}$.

\begin{remark}
\label{r2.2}
If the variable coefficients $a_{kj}$ and the function $\kappa$ in \eqref{2.1-w} and \eqref{2.2-w} belong to $C^2({\mathbb{R}}^3)$
and  $u^{inc}\in H^2_{loc}(\R^3)$, then
conditions \eqref{2.18-w} and \eqref{2.19-w} can be reduced to the homogeneous ones by introducing
a new unknown function   $\widetilde{u}:=u^{tot}-u^{inc}$  in $ \R^3$,
since $T^-_1u^{inc}=T^+_2u^{inc}$  on $S$.
For the  function   $\widetilde{u}$,
the above formulated  transmission problem is reduced then to the following one:
 \\
  \textit{Find a solution $\widetilde{u}\in  H^{2}_{loc}({\mathbb{R}}^3)\cap Z({\mathbb{R}}^3)$
to the differential equation
\begin{eqnarray}
\label{2.2-ww}
A\,\widetilde{u}(x)\equiv \partial_{x_k}\,\big(a_{kj}(x)\, \partial_ {x_j} \widetilde{u}(x)\big)
+\omega^2\,\kappa(x)\,\widetilde{u}(x)= \widetilde{f}(x),\;\;\;x\in {\mathbb{R}}^3,
\end{eqnarray}
where $\widetilde{f}:= f-Au^{inc}\in H^{0}_{comp}({\mathbb{R}}^3)$ due to the inclusions  $f\in H^{0}_{comp}({\mathbb{R}}^3)$ and $Au^{inc}=A_1u^{inc}=0$ in ${\Omega_1}$.
}

If $A\equiv \Delta+\omega^2\,\kappa(x)$ in ${\mathbb{R}}^3$ with $\kappa$ as in
\eqref{2.2-w}, then equation \eqref{2.2-ww} can be
 equivalently reduced to the Lippmann-Schwinger type integral equation (see, e.g. Colton \& Kress \cite{CK2}, Ch.8).

In our analysis, even for $C^2({\mathbb{R}}^3)$-smooth coefficients
we always will keep the transmission conditions \eqref{2.18-w}--\eqref{2.19-w} which
allow us to reduce the problem under consideration to the system of
localized boundary-domain integral
equations which live on the bounded domain $\Omega_2$
and its boundary $S$ (cf. N{\'e}d{\'e}lec \cite{Ned}, Ch. 2).
\end{remark}

Let us prove the uniqueness theorem for the transmission problem.
\begin{theorem}
\label{T2.4}
The homogeneous transmission problem \eqref{2.16-w}--\eqref{2.19-w} (with  $f_1=0,$ $f_2=0,$ $\varphi_0=\psi_0=0$)
possesses only the trivial solution.
\end{theorem}
\begin{proof} Denote by $B(R)$ a ball centred at the origin and radius $R$, $\Sigma_R:=\pa B(R)$.
We assume that $R$ is   a   sufficiently large positive number such that $\overline{\Omega}_2\subset B(R)$.
Let a pair $ (u_1, u_2)$ be a solution to the homogeneous transmission problem \eqref{2.16-w}--\eqref{2.19-w}.
Note that $u_1\in C^{\infty}(\Omega_1)$  due to  ellipticity of the constant coefficient operator $A_1$.
We can write  the first Green identities for the domains
${\Omega_2}$ and ${\Omega_1}(R) := {\Omega_1}\cap B(R)$  (see \eqref{2.14-wG} and \eqref{2.14-wG1}),
\begin{align}
&
\label{2.87-w}
\int_{{\Omega_2} }
  [a^{(2)}_{kj}(x)\,\pa_{j}u_2 (x)\; \overline{\pa_{k}  u_2(x)}-
  \omega^2\kappa_2(x) |u_2 (x)|^2]\,dx= \langle {  T^+_2}u_2 \,,\,\overline{\gamma^+u_2}\rangle_{S },\\[2mm]
  &
\label{2.88-w}
\int_{{\Omega_1}(R) }
  [a^{(1)}_{kj} \,\pa_{j}u_1 (x)\; \overline{\pa_{k}u_1 (x)}- \omega^2{ \kappa_1}|u_1 (x)|^2]\,dx
  = -\langle {  T^-_1}u_1\,,\,\overline{\gamma^-u_1} \rangle_{S }
  +\langle T^+_1u_1 \,,\,\overline{\gamma^-u_1}\rangle _{\Sigma(R)}.
  \end{align}
Since the matrices ${\bf a}_q=[a^{(q)}_{kj}]_{k,j=1}^3$ are symmetric and positive definite, in view of the homogeneous transmission conditions \eqref{2.18-w} and \eqref{2.19-w}, after adding \eqref{2.87-w} and \eqref{2.88-w} and taking the imaginary part,  we get
\[
\mbox{\rm Im}\;\Big\{
\int_{\Sigma_R}\ov{u_1(x)}\,T_{1}(x,\pa_x) u_1(x)\,d\Sigma_R\Big\}=0.
\]
Whence by Lemma \ref{L2.1} we deduce that $u_1=0$ in ${\Omega_1}$. In view of
\eqref{2.18-w}--\eqref{2.19-w} then we see that the function $u_2$ solves the homogeneous
Cauchy  problem in ${\Omega_2}$ for the elliptic partial differential
equation $A_2u_2=0$ with variable coefficients $a^{(2)}_{kj}$  and
$ \kappa_2$ being $C^2(\overline{\Omega}_2)$-smooth functions, see \eqref{2.3-w}.
By the interior and boundary regularity properties of solutions to elliptic problems we have
$u_2\in C^2(\overline{\Omega}_2)$ and
therefore $u_2=0$ in ${\Omega_2}$ due to the well known uniqueness theorem for the Cauchy problem (see, e.g.,
Landis \cite{Land}, Theorem 3; Calderon\cite{Cal},  Theorem 6).
\end{proof}
\begin{remark}
Due to the recent results concerning the Cauchy problem for scalar elliptic operators
one can reduce the smoothness of coefficients $a^{(2)}_{kj}$ and $ \kappa_2$ to the Lipschitz continuity and
require that ${\Omega_2}$ is a Dini domain,
see, e.g., Theorem 2.9 in Tao et al \cite{TZ}.
\end{remark}

\section{Reduction to   LBDIE  system  and equivalence theorem}
\label{ss2.2}
 \subsection{Integral relations in  the nonhomogeneous bounded domain} 

As it has already been mentioned, our goal is to reduce the above stated transmission problem
to the corresponding system of localized boundary-domain integral equations.
To this end let us define a \textit{localized parametrix} associated with  the fundamental solution
$-(\,4\,\pi\,|x|\,)^{-1}$ of the Laplace operator,
 \[
     P_{\chi} (x):=-\frac{\chi (x)}{4\,\pi\,|x|},
 \]
where $\chi$ is a cut off function $\chi \in X^4_+$,
see Appendix \ref{Appendix A}.
Throughout the paper we assume that this  condition
is satisfied  and
${\chi}$ has a compact support if not otherwise stated.

Let us consider Green's second identity for functions $u_2,v_2\in H^{1,\,0}(\Omega_2; A_2)$,
\[
\int _{\Omega_2{ (y,\varepsilon)}} \big(v_2 A_2 u_2 - u_2 A_2 v_2\big) \,dx
= \big\langle T^+_2 u_2,  \gamma^+v_2\big\rangle_{\partial\Omega_2(y,\varepsilon)}
- \big\langle\gamma^+u_2,T^+_2  v_2\big\rangle_{\partial\Omega_2(y,\varepsilon)}
\]
 where ${\Omega_2}(  y, \varepsilon):={\Omega_2}\setminus B(y,\varepsilon)$ with $B(y,\varepsilon)$ being
 a ball centred at the point $y\in {\Omega_2}$ with radius $\varepsilon>0$.
Substituting for $v_2(x)$ the parametrix $P_\chi(x-y)$, by
standard limiting arguments as $\varepsilon\to 0$ one can derive  Green's third identity
for $u\in H^{1,\,0}({\Omega_2},A_2)$ (cf. Chkadua et al \cite{CMN6}),
\begin{eqnarray}
\label{2.23-w}
 \beta\,u_2 + {\cal N}_\chi \,u_2- V_\chi T^+_2u_2 + W_\chi\gamma^+u_2
= {\cal P}_\chi  A_2u_2 \;\;\;\; \mbox{in}\;\;{\Omega_2},
\end{eqnarray}
where
\begin{align}
\label{2.32-w}
\dst
 \beta(y) =\frac{1}{3}\;\big[\,a^{(2)}_{11}(y)+a^{(2)}_{22}(y)+a^{(2)}_{33}(y)\,\big],
 \end{align}
$ {\cal N}_\chi $ is a singular localized integral operator which is understood in the Cauchy principal value sense,
 \begin{align}
\label{2.24-w}
\dst
{\cal N}_\chi\,u_2(y):=& {\rm v.p.}  \int _{{\Omega_2}  } [ A_2 (x,\pa_x)  P_\chi (x-y) ]  u_2(x) \,dx
=\lim\limits_{\varepsilon \to 0}\int _{{\Omega_2}(  y, \varepsilon)} [ A_2 (x,\pa_x)  P_\chi (x-y) ]\, u_2(x) \,dx,
\;\;y\in \mathbb{R}^3,
\end{align}
$ V_\chi$, $  W_\chi $, and $ {\cal P}_\chi  $
are the localized single layer, double layer, and Newtonian volume potentials respectively,
\begin{eqnarray}
\label{2.25-w}
 &
 \dst
  V_\chi \,g(y):=-\int _{S}  P_\chi ( x-y)\, g(x)\,dS_x,
\qquad
  W_\chi  \,g(y):=-\int _{S}\big[\,T_2(x, \partial_x)\, P_\chi (x-y)\,\big]\, \, g(x)\,dS_x,
 \;\; \;\;y\in\mathbb{R}^3\setminus S, \\
\label{2.27-w}
&
 \dst
 {\cal P}_{\chi}\,h(y):= \int _{{\Omega_2}}   P_\chi (x-y)\,h(x)\,dx, \;\;\;\;y\in\mathbb{R}^3.
\end{eqnarray}
{\bl Note that if $P_\chi$ is replaced with the corresponding fundamental solution, then ${\cal N}_\chi u_2=0$, $\beta=1$, and the third Green identity reduces to the familiar integral representation formula.}

 If the domain of integration in  \eqref{2.24-w} and  \eqref{2.27-w} is the whole
space ${\mathbb{R}}^3$, we employ the notation
\begin{eqnarray}
\label{2.28-w}
\dst
  {\bf N}_\chi \,h(y):= { \rm v.p.} \int _{\mathbb{R}^3} [\,A_2(x,\pa_x)  P_\chi (x-y)\,]\, h(x) \,dx\,,\;\;\;\;
  {\bf P}_\chi \,h(y):= \int _{\mathbb{R}^3}  P_\chi (x-y)\,h(x)\,dx,
\end{eqnarray}
 where the operator $A_2(x,\pa_x)$ in the first integral in \eqref{2.28-w} is assumed to be extended to the whole $\R^3$.
 Some mapping  properties of the above potentials needed  in our analysis are collected in  Appendix \ref{Appendix B}.

In view of the following distributional equality
\[
\dst
 \frac{\pa^2}{\pa x_k\,\pa x_j}\,\frac{1}{|x-y|}=
 - \frac{4\,\pi\, \delta_{kj}}{3}\;\delta(x-y)+ {\rm{v.p.}} \,\frac{\pa^2}{\pa x_k\,\pa x_j}\frac{1}{|x-y|},
\]
where $\delta_{kj}$ is the Kronecker delta and $\delta(\,\cdot\,)$ is the Dirac distribution,
we have (again in the distributional sense)
\begin{align}
\dst
A_2(x,\pa_x) P_\chi (x-y)&=a^{(2)}_{kj}(x)\frac{\pa^2   P_\chi (x-y)}{\pa x_k\,\pa x_j} +
 \frac{a^{(2)}_{kj}(x)}{\pa x_k}\,\frac{\pa  P_\chi (x-y)}{ \pa x_j}+ \omega^2\kappa_2(x)  P_\chi (x-y) \nonumber\\
  &=  \beta(x) \;\delta(x-y)+ {\rm{v.p.}}\;A_2 (x,\pa_x) P_\chi (x-y),
 \label{2.30-w}
 \end{align}
where
\begin{align}
 \label{2.35-w}
 \dst
 {\rm{v.p.}}\,A_2(x,\pa_x)  P_\chi (x-y)&=   {\rm{v.p.}} \Big[-\frac{a^{(2)}_{kj}(x)}{4\,\pi}
  \frac{\pa^2}{\pa x_k\,\pa x_j}\,\frac{1}{|x-y|}\Big]
 +R(x,y)
 \dst
 =   {\rm{v.p.}} \Big[-\frac{a^{(2)}_{kj}(y)}{4\,\pi} \frac{\pa^2}{\pa x_k\,\pa x_j}\,\frac{1}{|x-y|}\Big]
  +\widetilde{R} (x,y),
\end{align}
\begin{align*}
 &
 \dst
 R(x,y):= -\frac{1}{4\,\pi}\,\Big\{
 \frac{\pa}{\pa x_k}\,\Big[ \frac{\pa \chi(x-y)}{ \pa x_j}\,\frac{a^{(2)}_{kj}(x)}{|x-y|}\,\Big]
 + \frac{\pa \big[ a^{(2)}_{kj}(x) \, \chi(x-y)\big]}{\pa x_k}\, \,\frac{\pa}{ \pa x_j}\frac{1}{|x-y|}  \nonumber\\
 &
 \hskip17mm
 \dst
 +a^{(2)}_{kj}(x) \, \big[\chi(x-y)-1\big]\, \frac{\pa^2}{\pa x_k \pa x_j}\frac{1}{|x-y|} \Big\}
 + \omega^2\kappa_2(x)  P_\chi (x-y)\,,\\
&
 \dst
  \widetilde{R} (x,y):=R(x,y)-\frac{a^{(2)}_{kj}(x)-a^{(2)}_{kj}(y)}{4\,\pi} \frac{\pa^2}{\pa x_k\,\pa x_j}\,\frac{1}{|x-y|}\,.
 \end{align*}
Since  $\chi(0)=1$,  the functions $R(x,y)$ and $\widetilde{R}(x,y)$ possess  weak singularities of type
${\cal O}(|x-y|^{-2})$ as $x\to y$. However, the whole term ${\rm{v.p.}}\,A_2(x,\pa_x)  P_\chi (x-y)$ possesses the strong Cauchy singularity as $x\to y$. Thus, although $P_\chi$ is a parametrix for the Laplace operator, it is not a parametrix for the operator $A_2$, and we will call it instead a {\it quasi-parametrix} for $A_2$.

It is evident that if $a^{(2)}_{kj}(x)=a_2(x)\delta_{kj}$, then the  terms in square brackets in
formula \eqref{2.35-w} vanish  and ${\rm{v.p.}}\,A_2(x,\pa_x) P_\chi (x-y)$ becomes
 a weakly singular kernel.

Using the integration by parts formula   in \eqref{2.24-w},  one can easily derive the following relation for
$u_2\in H^1({\Omega_2})$
\begin{eqnarray}
\label{2.35-ww-1}
&&
  {\cal N}_\chi \,u_2=-\beta\,u_2 -   W_\chi \gamma^+u_2 +  {\cal Q}_\chi \,u_2 \;\;\;\;
\mbox{in}\;\;\Omega_2,
\end{eqnarray}
where
\begin{eqnarray}
\label{2.35-ww-2}
 {\cal Q}_\chi  \,u_2(y):=-\int _{{\Omega_2}} a^{(2)}_{kl}(x)\,\frac{\pa   P_\chi (x-y)}{\pa x_l}\;
 \frac{\pa u_2(x)}{\pa x_k}\,dx=
\pa _{y_l} \, {\cal P}_\chi  \big( a^{(2)}_{kl}\, \pa_k u_2\big)(y)\,, \;\;\;\; \forall\,y\in {\Omega_2}.
\end{eqnarray}
From  Green's third identity \eqref{2.23-w} and  Theorem  \ref{tB.4}
we deduce
\begin{eqnarray}
\label{2.35-ww-3}
\beta\,u_2 + {\cal N}_\chi  \,u_2\in H^{1,\,0}({\Omega_2},\Delta) \;\;\text{for}\;\;u_2\in H^{1,\,0}({\Omega_2},A_2),
\end{eqnarray}
 which, in turn, along with  relations \eqref{2.35-ww-1} and \eqref{2.35-ww-2} implies
\[
  {\cal Q}_\chi \,u_2= \pa _{y_l} \, {\cal P}_\chi \big( a^{(2)}_{kl}\, \pa_k u_2\big) \in H^{1,\,0}({\Omega_2},\Delta)
\;\;\text{for}\;\;u\in H^{1,\,0}({\Omega_2},A_2).
\]

 In what follows, in our analysis we need the explicit expression of the principal homogeneous symbol
${\mathfrak{S}}_0(  {\bf N}_\chi ;y, \xi)$ of the singular integral operator $  {\bf N}_\chi $,
which due to \eqref{2.30-w}  and  \eqref{2.35-w}   reads as
\begin{align}
 {\mathfrak{S}}_0( {\bf N}_\chi ;y, \xi)&={\cal F}_{z\to \xi}\left(-{\rm{v.p.}}\Big[\frac{a^{(2)}_{kl} (y)}{4\,\pi}
\frac{\pa^2}{\pa z_k\,\pa z_l}\,\frac{1}{|z|}\,\Big]\right)
=-\frac{a^{(2)}_{kl} (y)}{4\,\pi}\;{\cal F}_{z\to \xi}\left({\rm{v.p.}}
\Big[\frac{\pa^2}{\pa z_k\,\pa z_l}\,\frac{1}{|z|}\,\Big]\right)
\nonumber
\\
&
= -\frac{a^{(2)}_{kl} (y)}{4\,\pi} \;{\cal F}_{z\to \xi}\Big[\frac{4\,\pi\, \delta_{kl}}{3}\, \delta(z)
+\frac{\pa^2}{\pa z_k\,\pa z_l}\,\frac{1}{|z|}  \,\Big]
=-\beta(y)- a^{(2)}_{kl} (y)(-i\,\xi_k)(-i\,\xi_l)\;{\cal F}_{z\to \xi}\Big[ \, \frac{1}{4\pi|z|}  \,\Big]
\nonumber\\
\label{2.34-2}
&
=-\beta(y)+   \frac{a^{(2)}_{kl} (y) \,\xi_k \,\xi_l}{|\xi|^2}=  \frac{ A_2(y,\xi)}{|\xi|^2}
 -\beta(y)\,,\;\;y\in \overline{\Omega}_2, \;\; \xi\in {\mathbb{R}}^3,
\end{align}
where $A_2(y,\xi)= a^{(2)}_{kl} (y) \,\xi_k \,\xi_l.$
Here and in what follows, ${\cal F}$ and ${\cal F}^{-1}$ denote  the distributional direct and inverse Fourier transform operators  which for a summable function $g$ read as
$$
{\cal F}_{z\to \xi}[\,g\,]= \int_{\mathbb{R}^n}g(z)\,e^{i\,z\cdot \xi}\, dz, \quad
{\cal F}_{\xi \to z}[\,g\,]= \frac{1}{(2\pi)^n}\int_{\mathbb{R}^n}g(\xi)\,e^{-i\,z\cdot \xi}\, d\xi .
$$
In derivation of formula \eqref{2.34-2}, we employed that
${\cal F}_{z\to \xi}[(4\pi|z|)^{-1}]=|\xi|^{-2}$ and
${\cal F}_{z\to \xi}[\pa_j g]=-i \xi_j {\cal F}_{z\to \xi}[g]$  for $n=3$.
\\
Note that the principal homogeneous symbol
$ {\mathfrak{S}}_0(  {\bf N}_\chi ;y, \xi)$ is a rational homogeneous even function of order zero in $\xi$.
\\
In view of Theorem \ref{tB.7},  the interior trace of equality \eqref{2.23-w} on $S$ reads as
 \begin{eqnarray}
\label{2.35-ww}
   \mathcal N^+ _\chi \,u_2 - \mathcal V_\chi T^+_2u_2 +
   [(\beta - \mu)\,I + \mathcal W_\chi \,]\gamma^+u_2
 =  {\cal P}^+_\chi A_2u_2  \;\;\text{on}\;\;   S,
\end{eqnarray}
where the functions $\beta$ and $\mu$ are defined by \eqref{2.32-w} and \eqref{3.8j},
 $ \mathcal N^+_\chi =\gamma^+  {\cal N}_\chi $, $ {\cal P}^+_\chi =\gamma^+ {\cal P}_\chi $,
  while  the operators
$ \mathcal V_\chi $ and $ \mathcal W_\chi $, generated by the direct values of the single and double layer potentials,  are given by formulas \eqref{B-d1}.

Finally, we formulate a  technical lemma  which follows from formulas
\eqref{2.35-ww-1}, \eqref{2.35-ww-2}, and Theorem \ref{tB.4}.
\begin{lemma}
\label{L2.3}
Let $\Phi \in H^{1,\,0}({\Omega_2}; \Delta)$,  $ \psi \in H^{-\frac{1}{2}}(S)$,
$\varphi\in H^{\frac{1}{2}}(S)$, $\chi\in X^3$,   and the function $\beta$ be defined by \eqref{2.32-w}. Moreover, let $u_2\in H^1({\Omega_2})$ and
the following equation hold
\[
\beta\,u_2 +  {\cal N}_\chi \,u_2-  V_\chi \, \psi +  W_\chi \, \varphi =\Phi  \quad
\text{in }\;\;\Omega_2.
\]
Then $u_2\in H^{1,\,0}({\Omega_2}; A_2)$
  and the following estimate holds for some constant $C>0$,
\begin{align*} 
\|u_2\|_{H^{1,0}(\Omega_2; A_2)}
\le
\,C\,\big(
\|u_2\|_{H^{1}(\Omega_2)}+
\|\psi\|_{H^{-\frac{1}{2}}(S)}+ \|\varphi\|_{H^{\frac{1}{2}}(S)}
+ \|\Phi \|_{H^{1,\,0}(\Omega_2; \Delta)}\big).
\end{align*}
\end{lemma}

\subsection{Integral relations in  the  homogeneous unbounded domain}
For any radiating solution $u_1\in   H^{1,\,0}_{loc}({\Omega_1},A_1)\cap Z({\Omega_1})$
with $  A_1u_1 \in H^0_{comp}({\Omega_1})$ there holds Green's third   identity (for details see the references
Colton \& Kress \cite{CK2}, Vekua \cite{Vek},  Jentsch et al \cite{JN2}, Natroshvili et al \cite{NKT})
\begin{eqnarray}
\label{2.36-w}
  u_1+ V_\omega T^-_1u_1 -W_\omega\gamma^-u_1={\cal P}_\omega  A_1 u_1 \quad\text{in}\quad \Omega_1,
\end{eqnarray}
where
\begin{eqnarray}
\label{2.37-w}
&
\dst
  V_\omega \,g(y):=-\int_S\Gamma(x-y,\omega)\,g(x)\,dS_x,
\qquad
  W_\omega \,g(y):=-\int_S [ T_1(x,\pa_x)\Gamma(x-y,\omega )]\,g(x)\,dS_x,
    \quad y\in \R^3\setminus S,
    \\
  &
  \dst
\label{2.39-w}
{\cal P}_\omega  \,f(y):=\int_{{\Omega_1}} \Gamma(x-y,\omega ) \,f(x)\,dx,\quad y\in {\mathbb{R}}^3.
\end{eqnarray}
Here  $T_1(x,\pa_x)=a^{(1)}_{kj}n_k(x) \pa_{x_j}$, $n(x)$ is the outward unit normal
vector to $S$ at the point $x\in S$, and
\be
\Gamma(x,\om)=-\frac {\mbox{\rm exp}
 \{i\om \kappa_1 ^{ {1/2}} ({ {\bf a}^{-1}_1}x\cdot x)^{1/2}\}}
 {4\pi (\det{\bf a}_1)^{1/2}({ {\bf a}^{-1}_1}x\cdot x)^{1/2}}
\label{2.40-w}
 \ee
 is a radiating  fundamental   solution   of the operator $A_1$
 (see, e.g., Lemma 1.1 in Jentsch et al \cite{JN2}).
If $x$ belongs to a bounded subset of $\R^3$, then for sufficiently large $|y|$ we have the following asymptotic formula                 
\be
 \label{2.41-w}
\Gamma(y-x,\om)=
c(\xi)\frac{\mbox{\rm exp}\{i\xi\cdot(y-x)\}} {|y|}
+\mathcal{O}(|y|^{-2}),\;\;\;c(\xi)=
-\frac{|{\bf a}_1 \,\xi|}{4\pi \om \kappa_1^{1/2} \,( \det{\bf a}_1)^{1/2}},
\ee
where $\xi=\xi(\eta) \in S_{\om}$ corresponds to the direction
$\eta=y/|y|$ and is given by \eqref{2.11-w}.
 The asymptotic formula \eqref{2.41-w} can
be differentiated  arbitrarily many  times with respect to $x$ and $y$.

The mapping  properties of these potentials and   the boundary operators generated by them  are collected in  Appendix
\ref{Appendix C}.

Evidently, the layer potentials $ V_\omega g$ and $  W_\omega g$ solve the homogeneous differential equation \eqref{2.16-w}, i.e.,
\begin{align}
 \label{2.57-ww-1}
A_1  V_\omega \,g=A_1  W_\omega \,g=0 \;\;\text{in}\;\;{\mathbb{R}}^3\setminus S,
\end{align}
while for $f_1\in H^0_{comp}(\Omega_1)$ the volume potential $ {\cal P}_\omega   f_1  \in H^2_{loc}({\mathbb{R}}^3)$
solves the following nonhomogeneous equation (see Lemma \ref{L.C1}(i))
\begin{align}
 \label{2.57-ww-2}
A_1 {\cal P}_\omega    \,f_1 =\left\{
\begin{array}{lll}
f_1 & \text{in}& {\Omega_1},\\
0 & \text{in}& {\Omega_2}.
\end{array}
\right.
\end{align}
The exterior trace and co-normal derivative of the third Green identity  \eqref{2.36-w} 
on $S$ read as (see Lemma \ref{L.C1}(ii))
\begin{align}
 \label{2.57-w}
 &
 \mathcal V_\omega T^-_1u_1+ \left(\frac{1}{2}I - \mathcal W_\omega\right)\gamma^-u_1=\gamma^-{\cal P}_\omega A_1 u_1 \;\;\;\text{on}\;\;\;S,\\
\label{2.58-w}
 &
 \left(\frac{1}{2}I + \mathcal W_\omega'\right) T^-_1u_1- \mathcal L_\omega  \gamma^-u_1= T^-_1{\cal P}_\omega  A_1 u_1 \;\;\;\text{on}\;\;\;S,
\end{align}
where the integral operators ${ \mathcal V_\omega}$, ${ \mathcal W_\omega}$, ${ \mathcal W_\omega'}$, and ${ \mathcal L_\omega}$ are defined in
 Appendix \ref{Appendix C} by formulas \eqref{2.45-w}--\eqref{2.48-w}.
 Note that the operators  ${ \mathcal V_\omega}$, $2^{-1}I -{ \mathcal W_\omega}$, $2^{-1}I
 +{ \mathcal W_\omega'}$, and ${ \mathcal L_\omega}$ involved in \eqref{2.57-w}--\eqref{2.58-w}
 are not invertible  for  resonant values of the frequency parameter $\omega$.
 The set of these resonant values is countable and consists of eigenfrequencies of the
 interior Dirichlet and Neumann boundary value problems for the operator $A_1$ in the bounded domain $\Omega_2$ {\bl (see Vekua \cite[Section 4]{Vek}, Colton \& Kress \cite[Ch. 3]{CK1}, 
 Chen \& Zhou \cite[Section 7.7]{CZ}).}
 Therefore to obtain Dirichlet-to-Neumann or Neumann-to-Dirichlet mappings for arbitrary values of the frequency parameter $\omega$ 
 {we apply the ideas of the so called {\it combined-field integral equations}, cf. Burton \& Miller \cite{BM}, Brakhage \& Werner \cite{BW}, Colton \& Kress \cite{CK1, CK2}, Leis \cite{Le}, Panich \cite{Pa}.}
 
Multiply  equation \eqref{2.57-w} by $-i\, \alpha$ with some fixed  positive $\alpha$ and add to equation \eqref{2.58-w} to obtain
\begin{align}
 \label{2.59-w}
 \mathcal K_\omega   T^-_1u_1 - \mathcal M_\omega\gamma^-u_1 =\Psi_\omega A_1 u_1  \;\;\;\text{on}\;\;\;S,
\end{align}
where
\begin{align}
\label{2.61-w}
 \mathcal K_\omega  g  &:= \Big(\frac{1}{2}\,I+  \mathcal W_\omega'  \,-i\,\alpha\,\mathcal V_\omega\Big)g=
 \big(T_1^+  - i\, \alpha \,\gamma^+\big)  V_\omega \,g  \;\;\;\text{on}\;\;\;S,
 \\
\label{2.62-w}
 \mathcal M_\omega  h  &:= \Big[ \mathcal L_\omega -i\,\alpha\,\Big(-\frac{1}{2}\,I+ \mathcal W_\omega \Big)\Big]h=
 \big(T_1^+  - i\, \alpha \,\gamma^+\big)  W_\omega \,h  \;\;\;\text{on}\;\;\;S,
\\
\label{2.60-w}
 \Psi_\omega  \,f_1&:= \big( T^-_1   - i\, \alpha \, \gamma^-\big) {\cal P}_\omega   \,f_1=
 \big(T^+_1  - i\, \alpha \, \gamma^+\big) \mathcal P_\omega \,f_1
 \;\;\;\text{on}\;\;\;S,
\end{align}
for $f_1\in H^0_{comp}(\Omega_1)$, $g\in H^{-\frac{1}{2}}(S)$, and $h\in H^{\frac{1}{2}}(S)$.

In view of Lemma \ref{L.C2},  from \eqref{2.59-w} we derive the following analogue of the Steklov-Poincar\'{e}
type relation for arbitrary $u_1\in H^{1,\,0}_{loc}({\Omega_1};A_1)\cap Z({\Omega_1})$
\begin{align}
 \label{2.63-w}
  T^-_1u_1= {\mathcal K}^{-1}_\omega  \big(  \mathcal M_\omega  \, \gamma^- u_1  + \Psi_\omega  A_1u_1 \big)   \;\;\;\text{on}\;\;\;S,
\end{align}
where ${\mathcal K}^{-1}_\omega: H^{-\frac{1}{2}}(S)\to H^{-\frac{1}{2}}(S) $ is the inverse to the operator
 $ {\mathcal K}_\omega: H^{-\frac{1}{2}}(S)\to H^{-\frac{1}{2}}(S)$.

\subsection{
Equivalent reduction to a system of integral equations.}
Let us set
\be
\label{2.64-w}
\varphi_1=\gamma^- u_1,  \;\;\;\varphi_2:=\gamma^+u_2,
\;\;\;\psi_1= T^-_1 u_1,\;\;\;\psi_2:=  T^+_2 u_2.
\ee
If a pair $(u_1,u_2)$ solves the transmission problem \eqref{2.16-w}-\eqref{2.19-w}, then
by notation \eqref{2.64-w} and relations \eqref{2.23-w}, \eqref{2.35-ww}, \eqref{2.59-w},
\eqref{2.36-w},  the following equations hold true:
\begin{align}
\label{2.71-w}
&
\beta \,u_2 +   {\cal N}_\chi \,u_2 -  V_\chi \,\psi_2 +  W_\chi \,\varphi_2 = {\cal P}_\chi
 \,f_2   \;\;\;\text{in}\;\;\; {\Omega_2},\\
\label{2.72-w}
&
 {\mathcal N}^+_\chi \,u_2 - {\mathcal V}_\chi \psi_2 +[(\beta - \mu)\,I + \mathcal W_\chi ]\varphi_2
= {\cal P}^+_\chi  \,f_2  \; \;\;\text{on}\;\; \;  S,\\
&
\label{2.73-w}
 {\mathcal K}_\omega   \psi_2 -  {\mathcal M}_\omega  \varphi_2 = \Psi_\omega \,f_1
+ {\mathcal K}_\omega   \psi_{_0} -  \mathcal M_\omega  \varphi_{_0}   \;\;\;\text{on}\;\;\;S,\\
&
\label{2.75-w}
\psi_2-\psi_1=\psi_{_0} \;\;\;\text{on}\;\;\;S,\\
&
\label{2.74-w}
\varphi_2-\varphi_1=\varphi_{_0} \;\;\;\text{on}\;\;\;S,\\
&
\label{2.76-w}
u_1 +  V_\omega \,\psi_1-  W_\omega \,\varphi_1 = {\cal P}_\omega  \,f_1 \;\;\;\; \text{in}\;\;\; {\Omega_1}.
\end{align}
Let us consider relations \eqref{2.71-w}-\eqref{2.76-w} as a LBDIE system  with respect to  the  unknowns
$
(u_2, \psi_2, \varphi_2,  \psi_1, \varphi_1, u_1)\in \mathbf{H},
$
where
\begin{align}
\label{2.78-w}
\mathbf{H}:=H^{1,0}({\Omega_2}; A_2)\times H^{-\frac{1}{2}}(S)\times H^{\frac{1}{2}}(S)\times H^{-\frac{1}{2}}(S)\times H^{\frac{1}{2}}(S)
\times \big( H^{1,0}_{loc}({\Omega_1}; A_1) \cap Z({\Omega_1})\big).
\end{align}
{\bl Note that if $P_\chi$ would be replaced with the corresponding fundamental solution, then we would have ${\cal N}_\chi u_2=0$, ${\cal N}_\chi^+ u_2=0$, $\beta=1$, and $\mu=1/2$ in \eqref{2.71-w}-\eqref{2.72-w}.
Thus the system could be split to the boundary integral equation system \eqref{2.72-w}-\eqref{2.74-w} and the representation formulas \eqref{2.71-w}, \eqref{2.76-w} for the functions $u_1$ and $u_2$ in the domains $\Omega_1$ and $\Omega_2$, respectively.}

  Let us prove the following equivalence theorem.
\begin{theorem}
\label{T2.3}
Let conditions \eqref{2.20-w1} hold.

(i) If a pair $(u_2, u_1)\in   H^{1,\,0}({\Omega_2};A_2)\times \big( H^{1,\,0}_{loc}({\Omega_1};A_1)\cap Z({\Omega_1})\big)$
solves  transmission  problem \eqref{2.16-w}--\eqref{2.19-w},
then the vector $(u_2, \psi_2, \varphi_2,  \psi_1, \varphi_1, u_1)\in \mathbf{H}$ with
 $\psi_q$ and $\varphi_q$, $q=1,2,$  defined by \eqref{2.64-w},
  solves  LBDIE system \eqref{2.71-w}--\eqref{2.76-w}.

(ii)  Vice versa, if  a vector $(u_2, \psi_2, \varphi_2,  \psi_1, \varphi_1, u_1)\in \mathbf{H}$
solves LBDIE  system \eqref{2.71-w}--\eqref{2.76-w},  then the pair
$(u_2, u_1)\in H^{1,\,0}({\Omega_1};A_1)\times \big( H^{1,\,0}_{loc}({\Omega_1};A_1)\cap Z({\Omega_1})\big)$
solves   transmission problem \eqref{2.16-w}--\eqref{2.19-w} and
 relations \eqref{2.64-w} hold true.
\end{theorem}
\begin{proof} (i) The first part of the theorem directly follows form
the formulation of the transmission problem  \eqref{2.16-w}--\eqref{2.19-w} and relations
\eqref{2.23-w}, \eqref{2.35-ww}, \eqref{2.36-w}, \eqref{2.59-w}.

(ii) Now let a vector  $(u_2, \psi_2, \varphi_2,  \psi_1, \varphi_1, u_1)\in \mathbf{H}$  solve
 system  \eqref{2.71-w}--\eqref{2.76-w}.
Taking the trace of \eqref{2.71-w} on  $S$  and  comparing  with  \eqref{2.72-w} lead to the equation
\begin{eqnarray}
\label{2.79-w}
\gamma^+u_2=\varphi_2 \;\;\text {on}\;\; S.
\end{eqnarray}
Further, since $u_2\in H^{1,\,0}(\Omega_2;A_2)$ we can write  Green's third identity  \eqref{2.23-w} which in view of \eqref{2.79-w}
can be rewritten as
\begin{align}
\label{2.80-w}
 \beta\,u_2 + {\cal N}_\chi \,u_2- V_\chi T^+u_2 + W_\chi\varphi_2
= {\cal P}_\chi  A_2u_2 \;\;\;\; \mbox{in}\;\;{\Omega_2}.
\end{align}
From \eqref{2.71-w} and \eqref{2.80-w}, it follows that
\[
 V_\chi  (T^+u_2-\psi_2)+  {\cal P}_\chi  \big(A_2u_2-f_2\big)=0\;\; \text{in} \;\; \Omega_2.
\]
Whence by Lemma 6.3 in Chkadua et al\cite{CMN3} we deduce
\begin{align}
\label{2.82-ww-1}
A_2u_2=f_2\;\;\text{in}\;\;\Omega_2,\qquad
T^+ u_2=\psi_2\;\;\text{on}\;\;S.
\end{align}
From equation \eqref{2.76-w} it follows that
\begin{align}
\label{2.82-ww-2}
A_1u_1=f_1\;\;\text{in}\;\;\Omega_1.
\end{align}
From \eqref{2.73-w}, \eqref{2.74-w}, and \eqref{2.75-w} we derive
\begin{align}
\label{2.82-ww-3}
{ \mathcal K}_\omega   \psi_1 -  {\mathcal M}_\omega \varphi_2-  \Psi_\omega f_1=0
 \;\;\;\text{on}\;\;\;S.
\end{align}
Now, let us  consider the function
\begin{align}
\label{2.83-w}
  w  :=   V_\omega \,\psi_1-  W_\omega \,\varphi_1 - {\cal P}_\omega    f_1  \;\;\;\; \text{in}\;\;\; \Omega_2 .
\end{align}
  In  view of the inclusion ${ \cal P}_\omega  \,f_1 \in H^2_{loc}({ \mathbb{R}}^3)$
it follows that $ \gamma^+\mathcal P_\omega  f_1= \gamma^-{\cal P}_\omega f_1$ and
$T^+_1{ \cal P}_\omega   f_1=T^-_1 {\cal P}_\omega f_1$ on $S$.
Whence due to \eqref{2.61-w}--\eqref{2.60-w}, \eqref{2.82-ww-3},  and Lemma \ref{L.C1}, we have
$w\in H^{1,0}({\Omega_2}; A_1)$ and
\begin{align*}
 \big(T_1^+
- i\, \alpha \,\gamma^+ \big) w  =&\,
\Big( \frac{1}{2}I + {\mathcal W}_\omega'  - i\,\alpha\, \mathcal V_\omega\Big) \psi_1 -
 \Big[ \mathcal L_\omega -i\,\alpha\,\Big(-\frac{1}{2} I+ {\mathcal W}_\omega \Big)\Big]\varphi_1 -
 \big(   T^-  - i\, \alpha \,\gamma^{-}\big) {\cal P}_\omega  f_1
 \\
  = &
 \mathcal K_\omega   \psi_1 -  \mathcal M_\omega  \varphi_1 -  \Psi_\omega  f_1 =0\;\; \text{on}\;\; S.
\end{align*}
Consequently, in view of  \eqref{2.57-ww-1} and \eqref{2.57-ww-2} we see that the function $  w $ solves
the homogeneous Robin type interior boundary value problem,
\[
  A_1 w =0\;\;\text{in}\;\; {\Omega_2}, \qquad T_1^+ w - i\, \alpha \,\gamma^+  w =0
\;\;\text{on}\;\;S.
\]
By Green's first identity \eqref{2.14-wG} for the operator $A_1$ we have
\[
 \int_{\Omega_2} \overline{  w (x)} A_1w(x)\ dx=
-\int_{\Omega_2} \big[\,a^{(1)}_{kj} \,\pa_{j}w (x)\;\overline{\pa_{k} w  (x)}-\omega^2 \kappa_1|w(x)|^2\,\big]\,dx
+\big\langle  T^+_1w\,,\, \overline{\gamma^+ w}\big\rangle_S,
\]
  and since for the real symmetric matrix $a^{(1)}_{kj}$ the function
  $  a^{(1)}_{kj} \,\pa_{j} w  (x)
  \; \overline{\pa_{k}  w   (x)}$ is also real-valued,
it follows that  $\gamma^+ w =0$ and $T^+_1 w =0$ on $S$  for real $\alpha\neq 0$.
Consequently, the function $w$ defined in \eqref{2.83-w} vanishes identically in ${\Omega_2}$ in view of the corresponding Green's third identity.
Due to the jump relations for the  layer potentials presented in Lemma \ref{L.C1}(ii) and
since ${\cal P}_\omega   f_1 \in H^2_{loc}({ \mathbb{R}}^3)$,
we have  from \eqref{2.76-w} and \eqref{2.83-w}  the following relations,
\begin{align}
\label{2.86-w}
\gamma^-u_1=\gamma^-u_1 +\gamma^+ w =\varphi_1, \;\;\;\;\;    T^-_1u_1=  T^-_1u_1 + T_1^+  w =\psi_1.
\end{align}
From  equations \eqref{2.75-w}--\eqref{2.74-w} and relations \eqref{2.79-w}, \eqref{2.82-ww-1},  \eqref{2.82-ww-2}, and \eqref{2.86-w}
it follows that  the pair $(u_2,u_1)$
solves the transmission problem \eqref{2.16-w} and
 relations \eqref{2.64-w} hold true.
\end{proof}

From uniqueness Theorem \ref{T2.4} and the equivalence Theorem \ref{T2.3}, the following assertion follows directly.

\begin{corollary}
\label{C2.6}
Let conditions \eqref{2.20-w1} be fulfilled.
Then the LBDIE  system \eqref{2.71-w}--\eqref{2.76-w} possesses at most one solution in the space $\mathbf{H}$
defined in \eqref{2.78-w}.
\end{corollary}

\section{Analysis of the LBDIO}
\label{S-4}

Let us rewrite the LBDIE system  \eqref{2.71-w}--\eqref{2.76-w}  in a more convenient form for our further purposes
\begin{align}
\label{4.1-w}
&
(\beta \, I +  {\bf N}_\chi)\,\mathring{E}\,u_2 -  V_\chi \,\psi_2 +  W_\chi \,\varphi_2
={ \cal P}_\chi    f_2 \;\;\;\text{in}\;\;\; {\Omega_2},\\
\label{4.2-w}
&
 {\bf N}_\chi^+ \,\mathring{E}\,u_2 - \mathcal V_\chi \psi_2 +[(\beta - \mu)\,I + \mathcal W_\chi ]\varphi_2
 = {\cal P}_\chi^+  f_2  \; \;\;\text{on}\;\; \;  S,\\
&
\label{4.3-w}
 \mathcal K_\omega   \psi_2 -  \mathcal M_\omega  \varphi_2 = \Psi_\omega  f_1
+ \mathcal K_\omega   \psi_0 -  \mathcal M_\omega  \varphi_0   \;\;\;\text{on}\;\;\;S,\\
&
\label{4.4-w}
\psi_2-\psi_1=\psi_0 \;\;\;\text{on}\;\;\;S,\\
&
\label{4.5-w}
\varphi_2-\varphi_1=\varphi_0 \;\;\;\text{on}\;\;\;S,\\
&
\label{4.6-w}
u_1 +  V_\omega \,\psi_1-  W_\omega \,\varphi_1 = {\cal P}_\omega  f_1 \;\;\;\; \text{in}\;\;\; {\Omega_1},
\end{align}
where $\mathring{E}=\mathring{E}_{\Omega_2}$ denotes the extension operator by zero from ${\Omega_2}$ onto ${\Omega_1}$, ${\bf N}_\chi$ is a
pseudodifferential operator given in \eqref{2.28-w},
 ${\bf N}_\chi^+  =\gamma^+  {\bf N}_\chi $ and ${\cal P}^+_\chi =\gamma^+ {\cal P}_\chi $.
Note that for a function $u_2\in H^1(\Omega_2)$ we have
$ \beta \, u_2  +   {\cal N}_\chi \,u_2= \big(\beta \, I +  {\bf N}_\chi \big)\mathring{E} u_2$ in ${\Omega_2}.$

It can easily be seen that if the unknowns $(u_2, \psi_2, \varphi_2)$  are   determined  from the first
three equations  of  system   \eqref{4.1-w}--\eqref{4.6-w}, then the unknowns $(\psi_1, \varphi_1, u_1)$
are  determined  explicitly from the last three equations   of the same system.
Therefore the main   task  is to investigate the matrix integral operator generated
by the left had side expressions in  \eqref{4.1-w}--\eqref{4.3-w}.

Let us rewrite the first three equations of the LBDIE system  \eqref{4.1-w}--\eqref{4.6-w} in  matrix form
$$
   \mathbf{M}  \,U =F,
$$
where
 $ U  :=(u_2, \psi_2, \varphi_2)^\top,$
$ F :=\big(F _1, F _2, F _3\big)^\top,$
$F _1:= {\cal P}_\chi  f_2 ,$
$F _2:= {\cal P}_\chi^+   f_2 ,$
$F _3:= \Psi_\omega  f_1 + \mathcal K_\omega  \psi_0 -  \mathcal M_\omega\, \varphi_0 ,$
\begin{align}
\label{4.12-w}
&
 \mathbf M :=\left[
\begin{array}{ccc}
r_{\Omega_2}( \beta\,I+{\bf N}_\chi )\mathring{E} \;\;\;\; & -r_{\Omega_2}  V_\chi  \;\;\;\;& r_{ \Omega_2}  W_\chi   \\
 {\bf N}_\chi^+  \mathring{E}     \;\;\;\; & - \mathcal{V}_\chi \;\;\;\;& (\beta - \mu)I + \mathcal{W}_\chi   \\
0 \;\;\;\; &   \mathcal K_\omega  \;\;\;\; & -  \mathcal M_\omega
\end{array}
\right]\,.
\end{align}
Let us introduce the spaces
\begin{align*}
&
 \mathbb{H} :=H^{1,\,0}({\Omega_2}; A_2)\times H^{-\frac{1}{2}}(S)\times H^{\frac{1}{2}}(S),
 &&
 \mathbb{F} :=H^{1,\,0}({\Omega_2}; \Delta)\times H^{\frac{1}{2}}(S)\times H^{-\frac{1}{2}}(S),\\
&
 \mathbb{X}  :=H^{ 1 }({\Omega_2})\times H^{ -\frac{1}{2}}(S)\times H^{ \frac{1}{2}}(S),
 &&
 \mathbb{Y}:=H^{ 1 }({\Omega_2})\times H^{ \frac{1}{2}}(S)\times H^{ -\frac{1}{2}}(S),
\end{align*}
Recall that for $\chi\in X_+^4$ the principal homogeneous symbol ${\mathfrak{S}}_0( {\bf N}_\chi ;y, \xi)$
of the operator
$ {\bf N}_\chi $ given by \eqref{2.34-2} is a rational homogeneous function of order zero in $\xi$.
 Therefore, applying the inclusion  \eqref{2.35-ww-3} and
 the mapping properties of the pseudodifferential operators
with rational type symbols (see, e.g.,  Hsiao \& Wendland\cite{HW}, Theorem 8.4.13) and
using Theorems  \ref{tB.4} and   \ref{tB.6}  we deduce that the operators
\begin{align}
&
\label{4.17-ww-1}
  \mathbf M \;:\; \mathbb H  \to   \mathbb F ,
  \\
&
\label{4.17-w}
  \mathbf M \;:\; \mathbb X  \to  \mathbb Y
\end{align}
are continuous for $\chi\in X_+^4$.
 Now we prove the  main theorem of this section.

\begin{theorem}
\label{T4.1}
Let $\chi\in X_+^4$. Operator \eqref{4.17-w} is invertible.
\end{theorem}
\begin{proof} Using  Lemma \ref{L.C2}, we can represent the matrix operator $  \mathbf M $ defined in \eqref{4.12-w} as a composition of two operators
\[
 \mathbf M = {\bf B}{\bf C} ,
\]
where
\begin{align}
\label{4.18-w}
&
  {\bf B }:=\left[
\begin{array}{ccc}
r_{ \Omega_2} ( \beta\,  I  +   {\bf N}_\chi )\mathring{E}\;\;\;  & r_{ \Omega_2} \big[- V_\chi
+  W_\chi  \mathcal M_\omega  ^{-1}   \mathcal K_\omega \big] \;\; & r_{ \Omega_2}   W_\chi   \\
 {\bf N}_\chi^+  \mathring{E}   \;\;\;     & - \mathcal V_\chi + \big[(\beta - \mu)I + \mathcal W_\chi \big]
 { \mathcal M}_\omega^{-1}  { \mathcal K_\omega} \;\;& (\beta - \mu)I +{ \mathcal W_\chi}  \\
0\;\;\;&  0\;\;& - { \mathcal M_\omega}
\end{array}
\right]\,,
\quad
 {\bf C}:=\left[
\begin{array}{ccc}
 I \;\; & 0 \;\;& 0  \\
                            0 & I  & 0    \\
0\;\;&   - {\mathcal{M}}^{-1}_\omega\,{ \mathcal K_\omega} \;\;& I
\end{array}
\right]\,.
\end{align}
Evidently, the   triangular matrix  operator
\[
  {\bf C }\;:\; H^{ 1 }({\Omega_2} )\times H^{ -\frac{1}{2}}(S)\times H^{ \frac{1}{2}}(S)
  \to H^{ 1 }({\Omega_2} )\times H^{ -\frac{1}{2}}(S)\times H^{ \frac{1}{2}}(S)
\]
is invertible.
Since the operator ${ \mathcal M_\omega} :  H^{\frac{1}{2}}(S)\to H^{-\frac{1}{2}}(S)$
is also invertible due to Lemma \ref{L.C2},
  from  \eqref{4.18-w} it follows that the  block-triangular matrix  operator
\[
  {\bf B }\;:\; H^{ 1 }({\Omega_2} )\times H^{ -\frac{1}{2}}(S)\times H^{ \frac{1}{2}}(S)
  \to H^{ 1 }({\Omega_2} )\times H^{ \frac{1}{2}}(S)\times H^{- \frac{1}{2}}(S)
\]
and, consequently  operator \eqref{4.17-w}  is invertible if and only if the following operator is invertible
\begin{align}
\label{4.22-w}
 & {\bf D }\;:\; H^{ 1 }({\Omega_2} )\times H^{ -\frac{1}{2}}(S)
  \to H^{ 1 }({\Omega_2} )\times H^{ \frac{1}{2}}(S),
\\
 \label{4.23-w}
&
  {\bf D }
=\big[{\bf D } _{kj} { \big]_{k,j=1}^2}
:=\left[
\begin{array}{cc}
r_{ \Omega_2} ( \beta\,  I  + { {\bf N}_\chi })\mathring{E}  \;\;\;& r_{ \Omega_2} \big[-{  V_\chi}+{ W_\chi}
{\mathcal M}_\omega^{-1}  { \mathcal K_\omega}\big]    \\
 { {\bf N}_\chi^+} \mathring{E}       \;\;\; & -{ \mathcal V_\chi}+ \big[(\beta - \mu)I +{ \mathcal W_\chi}\big]
 {\mathcal M}_\omega^{-1}  { \mathcal K_\omega}
\end{array}
\right] \,.
\end{align}
Further we apply the Vishik-Eskin approach, developed in Eskin\cite{Esk},  and establish that operator
\eqref{4.22-w} is invertible.
\\
The proof is performed  in  four steps.

 \textsf{Step 1.} Here we show that the operator
\begin{align}
\label{4.24-w}
{ {\bf D}_{11}=} r_{_{\!\Omega_2}}( \beta\,  I  + {  {\bf N}_\chi })\mathring{E} \,:\,H^{1}({\Omega_2})\to H^{1}({\Omega_2})
 \end{align}
 is Fredholm with zero index.

 In view of  \eqref{2.34-2}   the principal homogeneous symbol of the operator $ \beta\,  I  + {  {\bf N}_\chi }$ can be written  as
\begin{eqnarray}
\label{ssss}
{\mathfrak{S}}_0({ {\bf D}_{11}};y,\xi)=
{\mathfrak{S}}_0( \beta\, I  + { {\bf N}_\chi };y,\xi)=\frac{A_2(y,\xi)}{ \Delta(\xi)}=\frac{a^{(2)}_{kl}(y)\xi_k \xi_l}{|\xi|^2}> 0,
\;\;\;\Delta(\xi) :=|\xi|^2,\;\;\;\;y\in \overline{\Omega}_2,\;\;\;\xi\in \mathbb R^3\setminus{\{0\}}.
\end{eqnarray}
 Since  the  symbol ${\mathfrak{S}}_0({ {\bf D}_{11}};y,\xi)$ given by \eqref{ssss} is an even  rational homogeneous function of order $0$  in $\xi$ it follows that its  factorization index $\varkappa$   equals to zero
(see Eskin\cite{Esk}, $\S\, 6$ ).
Moreover,  the operator $ \beta\, I  +  {\bf N}_\chi $ possesses
the transmission property.
Therefore we can apply the theory of pseudodifferential operators  satisfying the transmission property
to deduce that operator \eqref{4.24-w} is Fredholm
(see  Eskin\cite{Esk}, Theorem 11.1 and  Lemma 23.9;   Boutet de Monvel\cite{BdM}).

To show that Ind$\,{ {\bf D}_{11}}  =0$
we use the fact that the operators  ${ {\bf D}_{11}}$  and  ${ {\bf D}_{11,t}}$, where
$$
{ {\bf D}_{11,t}}=r_{_{\!\Omega_2}}[\,(1-t)\, I +t\,( \beta\,  I +{ {\bf N}_\chi^+} )\,]\,\mathring{E},\quad t\in [0,1],
$$
  are homotopic.  Evidently ${ {\bf D}_{11,0}}= I$ and ${ {\bf D}_{11,1}}= { {\bf D}_{11}}$.
In view of  \eqref{2.34-2} and \eqref{ssss},
\begin{eqnarray*}
&
\dst
{\mathfrak{S}}_0({ {\bf D}_{11,t}};y,\xi)=\frac{(1-t)\Delta(\xi)+t\;A_2(y,\xi)}{\Delta(\xi)}>0
\end{eqnarray*}
for all $t\in [0,1]$, for all $y\in \overline{\Omega}_2$, and for all $\xi\in \mathbb R^3\setminus{\{0\}}$,
and consequently the operator  ${ {\bf D}_{11,t}}$ is elliptic.

Since ${\mathfrak{S}}_0({ {\bf D}_{11,t}};y,\xi)$ is rational, even, and homogeneous of order zero in $\xi$,
  we conclude that the operator
$
{ {\bf D}_{11,t}} \; : \; H^{1}({\Omega_2})\to H^{1}({\Omega_2})
$
is continuous Fredholm operator for all $t\in[0,1]$. Therefore Ind$\,{ {\bf D}_{11,t}}$ is the same for all $t\in[0,1]$.
On the other hand, due to the equality ${ {\bf D}_{11,0}}= I$ we get
$
  {\rm{Ind}} \,{ {\bf D}_{11}}={\rm{Ind}}\, { {\bf D}_{11,1}}= {\rm{Ind}}\, { {\bf D}_{11,t}}=\,
  {\rm{Ind}}\, { {\bf D}_{11,0}}=0.
$

 \textsf{Step 2.} Now we show that the  operator ${\bf D}$ defined by \eqref{4.22-w}--\eqref{4.23-w}  is Fredholm.   To this end,
we apply the local principle (see, e.g., Eskin\cite{Esk}, $\S\, 19$ and  $\S\, 22$).

Let $U_j$  be an open  neighbourhood  of a fixed point  ${\widetilde{y}}\in {\mathbb{R}}^3$
and let   $\psi_0^{(j)},\varphi_0^{(j)}\in \mathcal{D}(U_j)$ be such that
${\rm{supp}}\,\psi_0^{(j)}\cap {\rm{supp}}\,\varphi_0^{(j)}\neq {\varnothing}$
contains some open neighbourhood $U'_j\subset U_j$ of the point $y_0$. \\
Consider the operator $\psi_0^{(j)} {\bf D}\,\varphi_0^{(j)}.$\\
We separate two possible cases: 1) ${\widetilde{y}}\in {{\Omega_2}}$ and 2) ${\widetilde{y}}\in S$.\\
In the first case, when ${\widetilde{y}}\in {{\Omega_2}}$ we can choose a
 neighbourhood   $\overline{U}_j$ of the point  ${\widetilde{y}}$ such that
  $\overline{U}_j\subset{{\Omega_2}}$. Then the operator $\psi_0^{(j)} {\bf D}\,\varphi_0^{(j)}$ is equivalent
to the  operator $\psi_0^{(j)}{ {\bf D}_{11}}\,\varphi_0^{(j)}$,
where ${ {\bf D}_{11}}$ is defined by \eqref{4.24-w}. As we have already shown  in Step 1,
this operator is Fredholm with zero index.

In the second case, when ${\widetilde{y}}\in S$, we need to check that  the \v{S}apiro-Lopatinski\u{i} type
condition for the operator $ {\bf D}$ is fulfilled, i.e., we have to show that the so-called
{\it boundary symbol} which is constructed by means of  the principal homogeneous symbols of the
pseudodifferential operators involved in \eqref{4.23-w}
is nonsingular (see Eskin\cite{Esk}, $\S$12).
To write   the boundary symbol function explicitly,
we assume that the symbols  are ``frozen"  at the point ${\widetilde{y}}\in S$ considered as the origin ${O}\,'$ of some local
coordinate system. Denote by  ${\widetilde{a}^{(2)} _{kl}}({\widetilde{y}})$ the corresponding ``frozen" coefficients of the principal part of
 the differential operator $A_2(y,\pa_y)$ subjected to
a translation and an orthogonal transformation related to the  local co-ordinate system.
If the matrix of the transformation of the original co-ordinate system $Oy_1y_2y_3$
to the new one $O\,'\eta_1\eta_2\eta_3$ with $O\,'={\widetilde{y}}$
is an orthogonal matrix $\Lambda({\widetilde{y}}):=[\lambda_{kl}({\widetilde{y}})]_{3\times 3}$,
which transforms the outward unit normal vector $n^\top({\widetilde{y}})$
into the vector $\mathbf{e}_3=(0,0,-1)^\top$ (the outward unit normal vector to $\mathbb{R}^3_{+}$), i.e., $n^\top({\widetilde{y}})=\Lambda({\widetilde{y}})\,\mathbf{e}_3$,
then
$y={\widetilde{y}}+\Lambda({\widetilde{y}})\,\eta,$ $\nabla _y=\Lambda({\widetilde{y}}) \, \nabla_\eta$, and
\begin{align}
\label{d-3-0}
&
\lambda_{k3}({\widetilde{y}})=- n_k({\widetilde{y}}),
\quad
  \widetilde{a}^{(2)} _{kl}({\widetilde{y}})=\lambda_{pk}({\widetilde{y}}) \,
  a^{(2)} _{pq}({\widetilde{y}})\,\lambda_{ql}({\widetilde{y}})=\{\Lambda^\top({\widetilde{y}}) \,\mathbf{a}_2({\widetilde{y}}) \,\Lambda({\widetilde{y}})\}_{kl}, \;\;\;\;k,l=1,2,3.
  \end{align}
Evidently the matrix
$
\widetilde{\mathbf{a}}_2({\widetilde{y}})=[\widetilde{a}^{(2)} _{kl}({\widetilde{y}})]_{k,j=1}^3:= \Lambda^\top({\widetilde{y}}) \,\mathbf{a}_2({\widetilde{y}}) \,\Lambda({\widetilde{y}})
$
 is positive definite, since $\mathbf{a}_2({\widetilde{y}})$ is positive definite and  for arbitrary ${\widetilde{y}}\in S$ we have
\begin{eqnarray*}
&
\widetilde{{\beta}} ({\widetilde{y}}) =\frac{1}{3}\big[\,{\widetilde{a}^{(2)} _{11}}({\widetilde{y}})+
{\widetilde{a}^{(2)} _{22}}({\widetilde{y}})+{\widetilde{a}^{(2)} _{33}}({\widetilde{y}})\,\big]>0,
\qquad
  {\widetilde{a} ^{(2)} _{33}}({\widetilde{y}})   =\lambda_{p3} \,a^{(2)} _{pq}({\widetilde{y}})\,\lambda_{q3}=
a^{(2)} _{pq}({\widetilde{y}})\, n_p({\widetilde{y}}) \,n_q({\widetilde{y}})=2\,\widetilde{\mu}({\widetilde{y}})>0,
\\
&
 T_2({\widetilde{y}}, \pa_y) =a^{(2)} _{pl}({\widetilde{y}})\,n_p({\widetilde{y}})\,\pa_{y_l}
  =
n_p({\widetilde{y}})\,{a}^{(2)} _{pl}({\widetilde{y}})\,\lambda_{\,lq}({\widetilde{y}}) \, \pa_{\eta_q}
=
-\lambda_{\,p3}({\widetilde{y}}) \,{a}^{(2)} _{pl}({\widetilde{y}})\,\lambda_{\,lq}({\widetilde{y}}) \, \pa_{\eta_q}=-  {\widetilde{a}^{(2)} _{3q}}({\widetilde{y}}) \, \pa_{\eta_q}\,,
\end{eqnarray*}
due to  \eqref{d-3-0}  and \eqref{3.8j}.

Further  let us note that the layer potentials can be represented by means of the volume potential
(see, e.g. Chkadua et al\cite{CMN6})
\begin{align}
&
 \label{4.26-w}
  {  V_\chi} \,\psi(y)=-{ {\bf P}_\chi}(\gamma^*\psi)(y), \;\; y\in \mathbb{R}^3\setminus S,\\
  &
  \label{4.27-w}
  { W_\chi}\,\varphi(y)
 =- \pa_{y_j}\,{  V_\chi}(a^{(2)}_{kj} \,n_k \,\varphi)=\pa_{y_j}\,
 { {\bf P}_\chi}\big(\gamma^*(a^{(2)}_{kj} \,n_k \,\varphi) \big)(y)\,,
 \;\; y\in \mathbb{R}^3\setminus S,
\end{align}
 where  $\gamma^*: H^{\ha-t}(S)\to H^{-t}_S$, $t>1/2$, is the adjoint operator to the trace operator $\gamma$, i.e.,
$
\langle \gamma^*\psi \,,\, h\rangle_{\mathbb{R}^3}:= \langle \psi \,,\,
 \gamma \,h\rangle_{_{S}}
$
for all $h\in{\cal D}(\mathbb{R}^3)$.
Here $H^{-t}_S:=\{f\in H^{-t}(\R^3):{\rm supp} f\subset S\}$,
and $H^{-t}_S$ does not contain non-zero elements,  when $t\leqslant\ha$ (see Lemma 3.39 in McLean\cite{McL}, Theorem 2.10(i) in Mikhailov\cite{Mik3}).

In view of \eqref{4.26-w}--\eqref{4.27-w}, the operator $\mathbf{D}_{12}$ in \eqref{4.23-w} can be represented as
\begin{align}
\label{4.30-w}
\mathbf{D}_{12}&=-{  V_\chi}(\psi_2) +{ W_\chi}\big({\mathcal M}_\omega^{-1}  { \mathcal K_\omega}\psi_2\big)
= { {\bf P}_\chi}(\gamma^*\psi_2 )
+\pa_{y_j}\,{ {\bf P}_\chi}\big(\gamma^*(a^{(2)}_{kj} \,n_k \,{\mathcal M}_\omega^{-1}  { \mathcal K_\omega}\psi_2 )\big)
\end{align}
and its  principal homogeneous symbol   due to the above formulas
and Remark \ref{rC.3} in Appendix \ref{Appendix C} can be written as
 \begin{eqnarray}
  \dst
\label{4.31-w}
\mathfrak{S}({\bf D }_{12}; {\widetilde{y}}, \xi)\equiv {\bf R}_{12}({\widetilde{y}}, \xi):=
 -\frac{1}{|\xi|^2}+
\frac{i\xi_l {\widetilde{a}^{(2)} _{3l}}({\widetilde{y}})} {|\xi|^2} \,
 {2 \,{\mathfrak{S}}_0\big({\mathcal V_\omega};{\widetilde{y}}, \xi'\big)},
\;\;  \xi=(\xi',\xi_3),
 \;\;\xi'=(\xi_1,\xi_2)\in {\mathbb{R}}^2\setminus \{0\},
 \end{eqnarray}
since the principal homogeneous symbol of the operator ${ {\bf P}_\chi}$  reads as
$\mathfrak{S}_0({\bf P};\xi)=-{\cal F}_{z\to \xi}[(4\pi|z|)^{-1}]=-|\xi|^{-2}.$

Due to the Vishik-Eskin approach, now we have to construct the following matrix  associated with the principal
homogeneous symbols of the operators involved in  ${\bf D}$  at the  local  co-ordinate
system introduced above
\begin{eqnarray}
\label{4.32-w}
{\bf R}( {\widetilde{y}},\xi):=
\left[
\begin{array}{ll}
{\bf R}_{11}({\widetilde{y}}, \xi) \;\;\;\;& {\bf R}_{12}({\widetilde{y}}, \xi)\\
{\bf R}_{21}({\widetilde{y}}, \xi) \;\;\;\;& {\bf R}_{22}({\widetilde{y}}, \xi')
\end{array}
\right]\,,
\end{eqnarray}
where ${\bf R}_{11}( {\widetilde{y}}, \xi)$ is the principal homogeneous symbol
of the operator ${ {\bf D}_{11}}=\beta\, I  + {  {\bf N}_\chi }$,
\begin{align}
\label{4.33-w}
{\bf R}_{11}( {\widetilde{y}}, \xi)=   {\mathfrak{S}}_0({ {\bf D}_{11}};{\widetilde{y}},\xi)
\equiv  {\mathfrak{S}}_0( \beta\, I  + { {\bf N}_\chi };{\widetilde{y}},\xi)
=\frac{A_2(\xi)}{\Delta(\xi)}=\frac{ {\widetilde{a}^{(2)} _{kl}}({\widetilde{y}})\xi_k \xi_l}{|\xi|^2}> 0,
 \;\;\;\xi\in {\mathbb{R}}^3\setminus \{0\},
\end{align}
${\bf R}_{12}({\widetilde{y}}, \xi)$ is the principal homogeneous symbol
of   operator \eqref{4.30-w} and is given by \eqref{4.31-w},
${\bf R}_{21}({\widetilde{y}}, \xi)$ is the principal homogeneous symbol
of the operator $  {\bf N}_\chi $,
 \begin{equation}
\label{4.35-w}
 {\bf R}_{21}( {\widetilde{y}}, \xi) :=
 {\mathfrak{S}}_0(  {\bf N}_\chi ;{\widetilde{y}},\xi)
 =\frac{A_2(\widetilde{y},\xi)}{\Delta(\xi)}-\widetilde{\beta}(\widetilde{y})=
 \frac{ {\widetilde{a}^{(2)} _{kl}}({\widetilde{y}})\xi_k \xi_l-\widetilde{\beta}(\widetilde{y})\,|\xi|^2}{|\xi|^2}\,,
 \end{equation}
 ${\bf R}_{22}({\widetilde{y}}, \xi)$ is the principal homogeneous symbol
of the boundary operator $ {\bf D}_{22}$,
which due to  \eqref{4.23-w},  \eqref{B-d-1}, \eqref{B-d-2-1}, and \eqref{C-2} is written as
\begin{align}
&
\label{4.36-w}
    {\bf R}_{22}( {\widetilde{y}},\xi') :={\mathfrak{S}}_0\big( -{ \mathcal V_\chi}+
     [(\beta - \mu)\,I +{ \mathcal W_\chi} ]{\mathcal M}_\omega^{-1}
      { \mathcal K_\omega};{\widetilde{y}},\xi'\big) \nonumber\\
&
\hskip3mm
= -{\mathfrak{S}}_0 \big({ \mathcal V_\chi};{\widetilde{y}},\xi')+ \frac{1}{2}\,{\mathfrak{S}}_0
\big( (\beta - \mu)\,I +{ \mathcal W_\chi} ;{\widetilde{y}},\xi'\big)
{\mathfrak{S}}_0 \big({\mathcal M}_\omega^{-1};{\widetilde{y}},\xi'\big)  \nonumber\\
&
\hskip3mm
=-\frac{1}{2\,|\xi'|}-\Big[
 2\widetilde{\beta}({\widetilde{y}})- {\widetilde{a}^{(2)} _{33}}({\widetilde{y}})-i
 \sum\limits_{l=1}^{2} {\widetilde{a}^{(2)} _{3l}}({\widetilde{y}})\,\frac{\xi_l}{|\xi'|}
 \Big] {\mathfrak{S}}_0\big({ \mathcal V_\omega};{\widetilde{y}}, \xi'\big).
 \end{align}
Below we drop the arguments ${\widetilde{y}}$   and  $\xi$ when it does not lead to misunderstanding.

Now we show that the \v{S}apiro-Lopatinski\u{i} type condition for the operator $ {\bf D}$ is satisfied,
i.e., the boundary symbol (see Eskin\cite{Esk}, $\S$12, formulas (12.25), (12.27))
\begin{eqnarray}
\label{4.39-w}
{\bf S}_{\mathbf{D}}( \xi')=-\Pi\,'\Big[\,\frac{{\bf R}_{21}}
{{\bf R}_{11}^{^{(+)}}}\;\Pi^+\Big(\frac{{\bf R}_{12}}
{{\bf R}_{11}^{^{(-)}}}\Big)\Big]( \xi'  )
+{\bf R}_{22}( \xi' )
\end{eqnarray}
associated with the operator ${\bf D}$
does not vanish for  $\xi' \neq 0$. Here,  ${\bf R}_{11} ^{^{(+)}}(\xi',\xi_3)$ and
${\bf R}_{11} ^{^{(-)}}(\xi',\xi_3)$ denote the ``plus" and ``minus"
factors respectively in the factorization of the symbol ${\bf R}_{11}(\xi',\xi_3)$ with respect
to the  variable $\xi_3$ in the complex $\xi_3$ plane, while $\Pi^+$
is  a Cauchy type integral operator
\[
\Pi^+(h)(\xi)=\frac{i}{2\pi}\;\lim\limits_{t\to 0+}\;
\int_{-\infty}^{+\infty}\frac{h(\xi',\eta_3)\,d\eta_3}{\xi_3+i\,t-\eta_3},
\]
and $\Pi\,'$ is the operator defined  on the set of rational functions
\[
\Pi'(g)(\xi')=-\frac{1}{2\pi}\int_{\ell^-}g(\xi',\xi_3)\,d\xi_3,
\]
where $\ell^-$ is a contour in the lower complex half-plane orientated counterclockwise and enclosing
all  poles of the rational  function $g$ with respect to $\xi_3$.

Denote the roots of the equation $A_2(\xi)\equiv {\widetilde{a}^{(2)} _{kl}}\xi_k\xi_l=0$ with respect to $\xi_3$
by  $\tau(\xi')=\alpha_1-i\,\alpha_2$ and
$\overline{\tau(\xi')}=\alpha_1+i\,\alpha_2$, where we assume that  $\alpha_2>0$.
Then
\begin{align}
\label{4.40-w}
&
A_2(\xi)= {\widetilde{a}^{(2)} _{kl}}\xi_k\xi_l={\widetilde{a}^{(2)} _{33}}\,[\xi_3-\tau(\xi')]\,[\xi_3-\overline{\tau (\xi')}]
=A^{^{(+)}}_2(\xi)\,A^{^{(-)}}_2(\xi),\\
 \label{4.41-w}
&
A^{^{(+)}}_2(\xi):={\widetilde{a}^{(2)} _{33}}\,[\xi_3-\tau(\xi')],\;\;
A^{^{ (-)}}_2(\xi):=\xi_3-\overline{\tau(\xi')}, \\
&
\label{4.41-w-1}
\overline{\tau(\xi')}= \alpha_1(\xi')+i\,\alpha_2(\xi'),\;\;\;\;\alpha_2(\xi')>0,
 \;\;\;\xi'\in \mathbb{R}^2\setminus \{0\}.
  \end{align}
Since
 $
  \Delta(\xi)=|\xi|^2= \Delta\!{^{^{(+)}}}(\xi)\, \Delta\!{^{^{(-)}}}(\xi)$ with
  $\Delta\!{^{^{(\pm)}}}(\xi):=\xi_3 \pm i\,|\xi'|, $
 we get the following factorization of the symbol ${\bf R}_{11}(\xi)$,
\begin{equation}
\label{4.43-w}
{\bf R}_{11}(\xi)={\bf R}_{11}^{^{(+)}}(\xi)\,{\bf R}_{11}^{^{(-)}}(\xi),\;\; \;\;
 {\bf R}_{11}^{^{(+)}}(\xi):=\frac{A_2^{^{(+)}}(\xi)}{\Delta\!{^{^{(+)}}}(\xi)},\;\;
{\bf R}_{11}^{^{(-)}}(\xi):=\frac{A_2^{^{(-)}}(\xi)}{\Delta\!{^{^{(-)}}}(\xi)} .
\end{equation}
Using formulas \eqref{4.33-w}--\eqref{4.36-w} and  \eqref{4.40-w}--\eqref{4.43-w}, we rewrite \eqref{4.39-w}  as
  \begin{align}
{\bf S}_{\bf D}( \xi')=&\,-\Pi\,'\Big\{\,\Big(  \frac{A_2 (\xi)}{\Delta (\xi)}-\widetilde{\beta} \Big) \frac{\Delta\!{^{^{(+)}}}(\xi)}{A_2^{^{(+)}}(\xi)}
\;\Pi^+\Big[\Big(-\frac{1}{\Delta (\xi)}-\frac{i\,\,\xi_l {\widetilde{a}^{(2)} _{3l}}}{\Delta (\xi)}[-2{\mathfrak{S}}_0\big({\mathcal V_\omega};\xi'\big)]\Big)
\frac{\Delta\!{^{^{(-)}}}(\xi)}{A_2^{^{(-)}}(\xi)}\Big]  \Big\}
\nonumber\\
\label{4.44-w}
&\,-\frac{1}{2\,|\xi'|}+\frac{1}{2}\Big( [2\widetilde{\beta}-{\widetilde{a}^{(2)} _{33}}]-i\sum\limits_{l=1}^{2}
{\widetilde{a}^{(2)} _{3l}}({\widetilde{y}})\,\frac{\xi_l}{|\xi'|}
 \Big)\,[-2 {\mathfrak{S}}_0\big({\mathcal V_\omega}; \xi'\big)]
 = {\bf S}_{\bf D}^{(1)}( \xi')+{\bf S}_{\bf D}^{(2)}( \xi')\;[-2{\mathfrak{S}}_0\big({ \mathcal V_\omega};\xi'\big)],
\end{align}
where
  \begin{align}
  \label{4.45-w}
{\bf S}_{\bf D}^{(1)}( \xi'):=&\,
 -\Pi\,'\Big[\,\Big(  \frac{A_2 (\xi)}{\Delta (\xi)}-\widetilde{\beta} \Big)
 \frac{\Delta\!{^{^{(+)}}}(\xi)}{A_2^{^{(+)}}(\xi)}
\;\Pi^+\Big(-\frac{1}{\Delta (\xi)}\,\frac{\Delta\!{^{^{(-)}}}(\xi)}{A_2^{^{(-)}}(\xi)}\Big)  \Big]
-\frac{1}{2\,|\xi'|}
 \nonumber\\
=&\,
 \Pi\,'\Big[\,\Big( \frac{A_2^{^{(-)}}(\xi)}{\Delta\!{^{^{(-)}}}(\xi)}-\widetilde{\beta}  \frac{\Delta\!{^{^{(+)}}}(\xi)}{A_2^{^{(+)}}(\xi)}\Big)
\;\Pi^+\Big(\frac{1}{\Delta\!{^{^{(+)}}}(\xi) A_2^{^{(-)}}(\xi)}\Big)  \Big]
-\frac{1}{2\,|\xi'|} \,,\\
  \label{4.46-w}
 {\bf S}_{\bf D}^{(2)}( \xi'):=&\, -\Pi\,'\Big[\,\Big(  \frac{A_2(\xi)}{\Delta (\xi)}-
 \widetilde{\beta} \Big) \frac{\Delta\!{^{^{(+)}}}(\xi)}{A_2^{^{(+)}}(\xi)}
\;\Pi^+\Big(-\frac{i\,\,\xi_l {\widetilde{a}^{(2)} _{3l}}}{\Delta
(\xi)}\,\frac{\Delta\!{^{^{(-)}}}(\xi)}{A_2^{^{(-)}}(\xi)}\Big)  \Big]
+\frac{1}{2} \Big[(2\widetilde{\beta}-{\widetilde{a}^{(2)} _{33}})
-i\sum\limits_{l=1}^{2}{\widetilde{a}^{(2)} _{3l}}\,\frac{\xi_l}{|\xi'|} \Big]
\nonumber\\
=&\,  \Pi\,'\Big[\,\Big(  \frac{A_2^{^{(-)}}(\xi)}{\Delta\!{^{^{(-)}}}(\xi)}-\widetilde{\beta } \frac{\Delta\!{^{^{(+)}}}(\xi)}{A_2^{^{(+)}}(\xi)}\Big)
\;\Pi^+\Big( \frac{i\,\,\xi_l {\widetilde{a}^{(2)}_{3l}}} {\Delta\!{^{^{(+)}}}(\xi) A_2^{^{(-)}}(\xi)}\Big)  \Big]
+\frac{1}{2} \Big[(2\widetilde{\beta}-{\widetilde{a}^{(2)} _{33}})
-i\sum\limits_{l=1}^{2}{\widetilde{a}^{(2)}_{3l}}\,\frac{\xi_l}{|\xi'|} \Big]
\end{align}
With the help of residue theorem, by direct calculations we find
 \begin{align}
\label{4.47-w}
&\Pi^+\Big(\frac{1}{\Delta\!{^{^{(+)}}} \,A_2^{^{(-)}}}  \Big)(\xi)
=\frac{i}{2 \pi}\;\lim\limits_{t\to 0+}
\int_{-\infty}^{+\infty}\frac{d\eta_3}{\Delta^{^{(+)}}(\xi',\eta_3) A_2^{^{(-)}}(\xi',\eta_3) (\xi_3+i\,t-\eta_3)}
 = \frac{i}{2 \pi} \lim\limits_{t\to 0+}
\int_{-\infty}^{+\infty}\frac{d\eta_3}{(\eta_3+i |\xi'|)\,(\eta_3-\overline{\tau(\xi')}) (\xi_3+i\,t-\eta_3)}
\nonumber\\
&
\hskip1cm
= -\frac{i}{2\,\pi} \lim\limits_{t\to 0+}
\int_{\ell^-}\frac{d\zeta}{(\zeta+i\,|\xi'|) (\zeta-\overline{\tau(\xi')})\,(\xi_3+i\,t-\zeta)}
  =-\frac{i}{2\,\pi} \lim\limits_{t\to 0+}
\frac{2\,\pi i}{(-i\,|\xi'|-\overline{\tau(\xi')}) (\xi_3+i\,t+i\,|\xi'|)}
\nonumber\\
&
\hskip1cm =-\frac{1}{(i\,|\xi'|+\overline{\tau(\xi')})\,(\xi_3+i\,|\xi'|)}\,;
\\
\label{4.48-w}
 & \Pi'\Big[\Big(\frac{A_2^{^{(-)}}}{\Delta\!{^{^{(-)}}}}\Big)\,\Pi^+
\Big(\frac{1}{\Delta\!{^{^{(+)}}}\,A_2^{^{(-)}}} \Big)\Big](\xi')
 =\frac{1}{2\,\pi}\int_{\ell^-}\frac{\xi_3-\overline{\tau(\xi')}}{\xi_3-i\,|\xi'|}\,
\frac{d\xi_3}{(i\,|\xi'|+\overline{\tau(\xi')})\,(\xi_3+i\,|\xi'|)}
=\frac{1}{2\,\pi\,[i\,|\xi'|+\overline{\tau(\xi')}]}
\int_{\ell^-}\frac{\xi_3-\overline{\tau(\xi')}}{\xi_3^2+|\xi'|^2}\,d\xi_3
\nonumber\\
&
\hskip1cm
=\frac{1}{2\,\pi\,[i\,|\xi'|+\overline{\tau(\xi')}]}\int_{\ell^-}\Big[\frac{1}{\xi_3-i\,|\xi'|}
-\frac{i\,|\xi'|}{\xi_3^2+|\xi'|^2}
-\frac{\overline{\tau(\xi')}}{\xi_3^2+|\xi'|^2} \Big]\,d\xi_3
=-\frac{1}{2\,\pi}\int_{\ell^-}\frac{d\xi_3}{\xi_3^2+|\xi'|^2}=\frac{1}{2\,|\xi'|},
\\
&
\label{4.49-w}
\widetilde{\beta}\,\Pi'\Big[\Big(\frac{\Delta^{^{(+)}}}{A_2^{^{(+)}}}\Big) \Pi^+
\Big(\frac{1}{ \Delta^{^{(+)}} A_2^{^{(-)}}} \Big)\Big](\xi')
=\frac{\widetilde{\beta}}{2 \pi}\int_{\ell^-}\frac{\xi_3+i\,|\xi'|}
{{\widetilde{a}^{(2)}_{33}} [\xi_3-\tau(\xi')]}\;
\frac{d\xi_3}{(i |\xi'|+\overline{\tau(\xi')})\,(\xi_3+i\,|\xi'|)}
 =\frac{\widetilde{\beta}}{2 \pi {\widetilde{a}^{(2)}_{33}}
[i |\xi'|+\overline{\tau(\xi')}]}\int_{\ell^-}\frac{d\xi_3}{\xi_3-\tau(\xi')}
\nonumber
\\
&
\hskip1cm=\frac{\widetilde{\beta}}{2\,\pi\,{\widetilde{a}^{(2)}_{33}}\,[i\,|\xi'|+\overline{\tau(\xi')}]}
\int_{\ell^-}\Big[\frac{1}{\xi_3+\tau(\xi')}+
\frac{2\,\tau(\xi')}{\xi_3^2-\tau^2(\xi')}\Big]\,d\xi_3
=\frac{\widetilde{\beta}}{2\,\pi\,{\widetilde{a}^{(2)}_{33}}\,[i\,|\xi'|+\overline{\tau(\xi')}]}
\int_{\ell^-}\frac{2\,\widetilde{\beta}\,\tau(\xi')}{\xi_3^2-\tau^2(\xi')}\;d\xi_3
\nonumber
\\
&
\hskip1cm=\frac{2\,\,\widetilde{\beta}\,\tau(\xi')}{2\,\pi\,
{\widetilde{a}^{(2)}_{33}}\,[i\,|\xi'|+\overline{\tau(\xi')}]}\;\;
\frac{2\,\pi\,i}{2\,\tau(\xi')}
=\frac{i\,\widetilde{\beta}}{{\widetilde{a}^{(2)}_{33}}\,[i\,|\xi'|+\overline{\tau(\xi')}]}\,.
\end{align}
Therefore from \eqref{4.45-w} in view of \eqref{4.47-w}--\eqref{4.49-w}  and \eqref{4.41-w-1} we get
\begin{align}
\label{4.50-w}
{\bf S}_{\bf D}^{(1)}( \xi')=-\frac{i\,\,\widetilde{\beta}}{{\widetilde{a}^{(2)}_{33}}\, [i\,|\xi'|+\overline{\tau(\xi')}]}
=-\frac{\,\widetilde{\beta}\,(\alpha_2+|\xi'|)+i\,\alpha_1\, \,\widetilde{\beta}}{{\widetilde{a}^{(2)}_{33}}\,[\alpha_1^2+(\alpha_2+|\xi'|)^2]}\;\;\;\text{for}\;\;\; \xi'\neq 0\,.
\end{align}
Now we  evaluate the function ${\bf S}_{\bf D}^{(2)}$.
Let
 $
  \vartheta(\xi'):=\sum\limits_{l=1}^{2}{\widetilde{a}^{(2)}_{3l}}\,\xi_l\,.
$
 Since $\tau$ and $\overline{\tau}$ are roots of the quadratic  equation
 $$
{ A_2(\xi)}\equiv \sum\limits_{k,l=1}^{3}{\widetilde{a}^{(2)} _{kl}}\xi_k\xi_l =
{\widetilde{a}^{(2)} _{33}} \xi_3^2+2\, \vartheta(\xi')\,\xi_3+
\sum\limits_{k,l=1}^{2}{\widetilde{a}^{(2)} _{kl}}\xi_k\xi_l=0,
$$
 we have
\begin{eqnarray}
 \label{4.52-w}
 2\, \vartheta(\xi') =-{\widetilde{a}^{(2)}_{33}}\,(\tau+\overline{\tau}).
 \end{eqnarray}
Again by direct calculations we find
\begin{align*}
 \Pi^+\Big( \frac{i\,\,\xi_l {\widetilde{a}^{(2)} _{3l}}}
{\Delta\!{^{^{(+)}}}(\xi) A_2^{^{(-)}}(\xi)}\Big)  \Big]
&=
\Pi^+\Big( \frac{i\,{\widetilde{a}^{(2)} _{3l}}\xi_l}{\xi_3+i\,|\xi'|}\,
\frac{1}{\xi_3-\overline{\tau}}\Big)(\xi')
=\frac{i}{2\,\pi}\;\lim\limits_{t\to 0+}\;
\int_{-\infty}^{+\infty}
\frac{[i\,\vartheta(\xi')+i\,{\widetilde{a}^{(2)} _{33}}\,\eta_3]\,d\eta_3}
 {(\eta_3+i\,|\xi'|)\,(\eta_3-\overline{\tau})\,(\xi_3+i\,t-\eta_3)}
\nonumber
\\
&
= -\frac{i\,\vartheta(\xi')+{\widetilde{a}^{(2)}_{33}}|\xi'|}{(\overline{\tau}+i\,|\xi'|)\,(\xi_3+i\,|\xi'|)}\,.
\end{align*}
Further we have
\begin{align}
&
 \Pi\,'\Big[\,\Big(  \frac{A_2^{^{(-)}}(\xi)}{\Delta\!{^{^{(-)}}}(\xi)}-\widetilde{\beta}  \frac{\Delta\!{^{^{(+)}}}(\xi)}{A_2^{^{(+)}}(\xi)}\Big)
\;\Pi^+\Big( \frac{i\,\,\xi_l {\widetilde{a}^{(2)}_{3l}}}
{\Delta\!{^{^{(+)}}}(\xi) A_2^{^{(-)}}(\xi)}\Big)  \Big]
  =\frac{1}{2\,\pi}\int _{\ell^-}
\Big(  \frac{A_2^{^{(-)}}(\xi)}{\Delta\!{^{^{(-)}}}(\xi)}-\widetilde{\beta}  \frac{\Delta\!{^{^{(+)}}}(\xi)}{A_2^{^{(+)}}(\xi)}\Big)
 \,\frac{i\,\vartheta(\xi')+{\widetilde{a}^{(2)}_{33}}|\xi'|}{(\overline{\tau}+i\,|\xi'|)\,(\xi_3+
 i\,|\xi'|)}\,d\xi_3
 \nonumber
 \\[2mm]
 \label{4.54-w}
 &
 \hskip1cm
=\frac{1}{2\,\pi}\,\frac{i\,\vartheta(\xi')+{\widetilde{a}^{(2)}_{33}}|\xi'|}{\overline{\tau}+i\,|\xi'|}
 \int_{\ell^-}\Big[
 \frac{\xi_3-\overline{\tau}}{\xi_3^2+|\xi'|^2}-\frac{\widetilde{\beta}}
 {{\widetilde{a}^{(2)}_{33}}\,(\xi_3-\tau)} \Big]\,d\xi_3
    =
      \frac{i\,\vartheta(\xi')+{\widetilde{a}^{(2)}_{33}}|\xi'|}{2\,|\xi'|}
         - \frac{i\,\widetilde{\beta}}{{\widetilde{a}^{(2)}_{33}}} \,\frac{i\,\vartheta(\xi')+{\widetilde{a}^{(2)}_{33}}|\xi'|}{\overline{\tau}+i\,|\xi'|}\,.
 \end{align}
Now, from \eqref{4.46-w}, \eqref{4.52-w}, and \eqref{4.54-w} we  get
\begin{align}
\label{4.55-w}
\dst
{\bf S}_{\bf D}^{(2)}( \xi')&\,=   \frac{i\,\vartheta(\xi')+
{\widetilde{a}^{(2)}_{33}}|\xi'|}{2\,|\xi'|}
-\frac{i\,\widetilde{\beta}}{{\widetilde{a}^{(2)}_{33}}}
\,\frac{i\,\vartheta(\xi')+{\widetilde{a}^{(2)}_{33}}|\xi'|}{\overline{\tau}+i\,|\xi'|}+
\frac{1}{2}\,\Big[ 2\,\widetilde{\beta}-{\widetilde{a}^{(2)}_{33}}-
i \,\frac{\vartheta(\xi')}{|\xi'|} \Big]
\nonumber
\\[2mm]
\dst
&\,=\frac{\widetilde{\beta}(\vartheta(\xi')+  {\widetilde{a}^{(2)}_{33}}\overline{\tau})}
{{\widetilde{a}^{(2)}_{33}} (\overline{\tau}+i\,|\xi'|)} = \frac{ \widetilde{\beta}(\overline{\tau}-  \tau)}
{2\,(\overline{\tau}+i\,|\xi'|)} =\frac{i \,\widetilde{\beta}\, \alpha_2 }
{\overline{\tau}+i\,|\xi'|}\;\;\text{for}\;\; \xi'\neq 0\,.
\end{align}
Finally, from \eqref{4.44-w} in view of \eqref{4.50-w}  and \eqref{4.55-w} we have
\begin{align*}
\dst
{\bf S}_{\bf D}( \xi')&\,=-\frac{\,\widetilde{\beta}\,(\alpha_2+|\xi'|)+
i\,\alpha_1\, \,\widetilde{\beta}}{{\widetilde{a}^{(2)}_{33}}\,
[\alpha_1^2+(\alpha_2+|\xi'|)^2]}
+\frac{i \,\widetilde{\beta}\, \alpha_2 } {\overline{\tau}+i\,|\xi'|}
\;\big[-2{\mathfrak{S}}_0\big({\mathcal V_\omega} ;\xi'\big)\big]
 \nonumber
 \\[3mm]
\dst
&\,=- \frac{\widetilde{\beta}\,(\alpha_2+|\xi'|)[1+2\,\alpha_2\,
{\widetilde{a}^{(2)}_{33}}\,{\mathfrak{S}}_0 ({\mathcal V_\omega} ;\xi')]+
 i\,\alpha_1\, \,\widetilde{\beta} \,[1-2\,\alpha_2\,
 {\widetilde{a}^{(2)}_{33}}\, {\mathfrak{S}}_0 ({\mathcal V_\omega} ;\xi')]}
{{\widetilde{a}^{(2)}_{33}}\,[\alpha_1^2+(\alpha_2+|\xi'|)^2]},
\end{align*}
whence the following inequality follows
\begin{align}
\label{4.57-w}
\dst
\mbox{\rm Re}\,\,{\bf S}_{\bf D}( \xi')
 =- \frac{\widetilde{\beta}\,(\alpha_2+|\xi'|)\,[1+2 \,\alpha_2\,{\widetilde{a}^{(2)}_{33}}
 \,{\mathfrak{S}}_0 ({\mathcal V_\omega} ;\xi')]}
{{\widetilde{a}^{(2)}_{33}}\,[\alpha_1^2+(\alpha_2 + |\xi'|)^2]}<0 \;\;\text{for}\;\; \xi'\neq 0\,
\end{align}
due to the relations (see \eqref{C-2})
\begin{eqnarray}
\label{4.58-w}
\widetilde{\beta}>0,\;\;\;{\widetilde{a}^{(2)}_{33}}>0,\;\;\; |\xi'|>0,\;\;\;\alpha_2>0, \;\;\;\;
\mathfrak{S}_0 ({\mathcal V_\omega};\xi')>0 \;\;\;\; \forall\, \xi'\neq 0.
\end{eqnarray}
Thus the \v{S}apiro-Lopatinski\u{i} type condition for the  ``boundary symbol" ${\bf S}_{\bf D}$ defined by
\eqref{4.39-w} is satisfied and  the operator $ {\bf D}$ in \eqref{4.22-w}--\eqref{4.23-w}   is Fredholm.

 \textsf{Step 3.} Here we prove that the index of the operator ${\bf D}$ equals to zero. To this end let us consider the operator
 \begin{align}
\label{4.59-w}
&
  {\bf D }_t:=\left[
\begin{array}{cc}
r_{_{\!\Omega_2}}( \beta\,  I  + { {\bf N}_\chi })\mathring{E}
&\;\;\; r_{_{\!\Omega_2}}\big[-{ V_\chi}+{  W_\chi}{\mathcal M}_\omega^{-1}  { \mathcal K_\omega}\big]
 \\[2mm]
t\; { {\bf N}_\chi^+} \mathring{E}
&\;\;\; (t-1)\beta\,I+t\,\big\{ - { \mathcal V_\chi}+ \big[(\beta - \mu)\,I +{ \mathcal W_\chi}\big]
{\mathcal M}_\omega^{-1}  { \mathcal K_\omega}\big\}
\end{array}
 \right]
\end{align}
with   $t\in [0,1]$, and  establish  that it is homotopic to the operator ${\bf D}$.

 Evidently, $ {\bf D }_1={\bf D}$ and
$ {\bf D}_t\; : H^{1}(\Omega_2)\times H^{-\frac{1}{2}}(S)\to H^{1}(\Omega_2)\times {H}^{\frac{1}{2}}(S)$.
First we show  that for the operator $ {\bf D }_t$ the \v{S}apiro-Lopatinski\u{i}
condition is satisfied for all $t\in [0,1]$.
The counterpart of the matrix \eqref{4.32-w} now reads as
\begin{eqnarray*}
{\bf R}_t( {\widetilde{y}},\xi):=
\left[
\begin{array}{cc}
{\bf R}_{11}({\widetilde{y}}, \xi) & \;\;\;\;{\bf R}_{12}({\widetilde{y}}, \xi)\\
t\;{\bf R}_{21}({\widetilde{y}}, \xi) &\;\;\;\; {\bf R} _{22,t}({\widetilde{y}}, \xi')
\end{array}
\right]\,,
\end{eqnarray*}
where ${\bf R}_{11}$, ${\bf R}_{12}$, and ${\bf R}_{21}$ are defined by formulas \eqref{4.33-w},
\eqref{4.31-w}, and \eqref{4.35-w}
respectively, while  in accordance with  \eqref{4.59-w} and \eqref{4.36-w}
\begin{align*}
    {\bf R} _{22,t}( {\widetilde{y}},\xi') :=&{\mathfrak{S}}_0\Big( (t-1)\beta\,I+t\,
    \big\{-{ \mathcal V_\chi}+ \big[(\beta - \mu)\,I +{ \mathcal W_\chi}\big]
    {\mathcal M}_\omega^{-1}  { \mathcal K_\omega}\big\};{\widetilde{y}},\xi'\Big)
=  t\, {\bf R}_{22}( {\widetilde{y}},\xi')+(t-1)\beta \,.
 \end{align*}
The corresponding boundary symbol associated with the  \v{S}apiro-Lopatinski\u{i} condition, the counterpart of \eqref{4.39-w}, has the form
\begin{align*}
{\bf S}_{{\bf D}, t}( \xi')&\,=-\Pi\,'\Big[\,\frac{t\,{\bf R}_{21}}
{{\bf R}_{11}^{^{(+)}}} \Pi^+\Big(\frac{{\bf R}_{12}}
{{\bf R}_{11}^{^{(-)}}}\Big)\Big]( \xi'  )
+{\bf R}^{(t)}_{22}( \xi' )
=-t \Pi\,'\Big[\,\frac{{\bf R}_{21}}
{{\bf R}_{11}^{^{(+)}}} \Pi^+\Big(\frac{{\bf R}_{12}}
{{\bf R}_{11}^{^{(-)}}}\Big)\Big]( \xi'  )
+t  {\bf R}_{22}( {\widetilde{y}},\xi') -(1-t) \widetilde{\beta}
 \nonumber\\
&
=t {\bf S}_{\bf D}( \xi')-(1-t) \widetilde{\beta} ,
\end{align*}
and due to the inequalities \eqref{4.57-w} and \eqref{4.58-w} we have
\[
\mbox{\rm Re}\,\,{\bf S}_{{\bf D},\,t}( \xi') = t\,\mbox{\rm Re}\,\,
{\bf S}_{\bf D} (\xi')-(1-t) \widetilde{\beta} <0 \;\;\;\;
 \forall\, \xi'\neq 0\;\;\;\;\forall \,t\in[0,\,1].
\]
Thus the  \v{S}apiro-Lopatinski\u{i}  condition for the operator $ {\bf D}_t$  is satisfied for all $t\in[0,\,1]$.
Therefore, as in the case of the operator ${\bf D}$, it follows that the operator
$
 {\bf D}_t\; : H^{1}(\Omega_2)\times H^{-\frac{1}{2}}(S)
\to
H^{1}(\Omega_2)\times {H}^{\frac{1}{2}}(S)
$
is Fredholm and has the same index  for all $t\in [0,1]$.

On the other hand, the upper triangular matrix operator $ {\bf D}_0$ has zero index since one of the operators in the main diagonal,
$
- \beta\, I  :
 H^{-\frac{1}{2}}(S)\to  H^{-\frac{1}{2}}(S)
$
is invertible, while the second operator,
 $
{ {\bf D}_{11}}=r_{_{\!\Omega_2}}( \beta\,  I  + {  {\bf N}_\chi })\mathring{E}    : H^{1}(\Omega_2)\to H^{1}(\Omega_2)
$
is Fredholm with zero index as it has been shown in Step 1.
Consequently,
$
{\rm{Ind}}\, {\bf D}={\rm{Ind}}\, {\bf D}_{\,1}={\rm{Ind}}\, {\bf D}_t={\rm{Ind}}\, {\bf D}_0=0\,.
$

 \textsf{Step 4.} Now we show that the operator ${\bf D}$ is injective which will imply its invertibility.

Let  $\widetilde{U}=(\widetilde{u}_2,\widetilde{\psi}_2)^\top\in H^{1}(\Omega_2)\times H^{-\frac{1}{2}}(S)$ be
a solution to the homogeneous equation
\begin{eqnarray}
\label{4.64-w}
 {\bf D}\,\widetilde{U}=0.
\end{eqnarray}
Since the operator ${\bf D}$ is Fredholm with zero index, there exists a left
regularizer $\mathfrak{R}_{\bf D}$ such that
\[
\mathfrak{R}_{\bf D} \;: H^{ 1}(\Omega_2)\times H^{ \frac{1}{2}}(S)
\to H^{ 1}(\Omega_2)\times {H}^{ -\frac{1}{2}}(S)
\]
and
$
\mathfrak{R}_{\bf D}\,  {\bf D}=I+\mathfrak{T}_{\bf D}\, ,
$
where    $\mathfrak{T}_{\bf D}$  is the operator of order  $-1$   (cf., e.g., the proof of
 Theorems 22.1 and 23.1 in Eskin\cite{Esk}),
\begin{eqnarray}
\label{4.67-w}
\mathfrak{T}_{\bf D}  \;: H^{ 1}(\Omega_2)\times H^{ -\frac{1}{2}}(S)
\to
H^{ 2}(\Omega_2)\times {H}^{ \frac{1}{2}}(S).
\end{eqnarray}
Therefore, for $\widetilde{U}=(\widetilde{u}_2,\widetilde{\psi}_2)^\top\in H^{1}(\Omega_2)\times H^{-\frac{1}{2}}(S)$ from  \eqref{4.64-w} we have
\begin{eqnarray}
\label{4.68-w}
 \mathfrak{R}_{\bf D}\,  {\bf D}\,\widetilde{U}=\widetilde{U}+\mathfrak{T}_{\bf D} \widetilde{U}=0
\end{eqnarray}
and, in view of   \eqref{4.67-w} and \eqref{4.68-w}, we deduce
$$
\widetilde{U}=(\widetilde{u}_2,\widetilde{\psi}_2)^\top\in H^{2}(\Omega_2)\times H^{\frac{1}{2}}(S).
$$
Clearly, by   $\widetilde{u}_2$ and $\widetilde{\psi}_2$ we can construct the vector
$U^{(0)}=(\widetilde{u}_2, \widetilde{\psi}_2, \widetilde{\varphi}_2,  \widetilde{\psi}_1, \widetilde{\varphi}_1, \widetilde{u}_1)\in \mathbf{H}$, a solution to the homogeneous system
\eqref{4.1-w}--\eqref{4.6-w}. Here $\mathbf{H}$ is defined in \eqref{2.78-w}.

Therefore by   equivalence Theorem \ref{T2.3} and    uniqueness Theorem \ref{T2.4} we conclude that
$U^{(0)}$ is  a zero vector. Thus the null space of the operator ${\bf D} $ is trivial
 in  the   class    $H^{1}(\Omega_2)\times H^{-\frac{1}{2}}(S)$.  Consequently,
the operator
$
 {\bf D}\;  : H^{1}(\Omega_2)\times H^{-\frac{1}{2}}(S)
\to
H^{1}(\Omega_2)\times {H}^{\frac{1}{2}}(S)
$
is invertible implying that the operator \eqref{4.17-w} is invertible as well which completes the proof.
\end{proof}

For a cut off function $\chi$ of infinite smoothness we have the following result.
\begin{corollary}
\label{cor-C4.2}
Let a cut-off function $\chi\in X_+^\infty$ .   Then the  operators
\begin{align}
&
\label{4.69-w}
 {\bf D}\;: \; H^{r+1}(\Omega_2)\times H^{r-\frac{1}{2}}(S) \to H^{r+1}(\Omega_2)\times {H}^{r+\frac{1}{2}}(S),\\
 &
 \label{4.70-w}
 {\bf M}\;: \; H^{r+ 1 }(\Omega_2)\times H^{r -\frac{1}{2}}(S)\times H^{r+ \frac{1}{2}}(S) \to
               H^{ r+1 }(\Omega_2)\times H^{ r+\frac{1}{2}}(S)\times H^{ r-\frac{1}{2}}(S),
\end{align}
where the ${\bf D}$ and  ${\bf M}$ are defined by \eqref{4.23-w} and \eqref{4.12-w}, respectively,
are invertible for all $r>-\ha$.
\end{corollary}
 \begin{proof} It can be carried out by the word for word arguments applied in the proof of Theorem \ref{T4.1}
and using the counterparts of Theorems  \ref{tB.4} and \ref{tB.6} describing
the mapping and smoothness properties of the localized potentials for a cut off function  of infinite smoothness which
actually coincide with the properties of usual potentials without localization.
\\
\indent
In the final part, Step 4, one needs to apply the fact that the operator \eqref{4.69-w}
possesses a common regularizer for all $r> -\ha$ (see, e.g., Agranovich\cite{Ag1}) implying that the null space of the operator
$ {\bf D}$ is trivial for  all   $r>-\ha$
which yields that the  operators \eqref{4.69-w} and
\eqref{4.70-w} are invertible   for   all   $r> -\ha$.
\end{proof}

 From  Theorem \ref{T4.1} and Lemma \ref{L2.3}, we derive also the  invertibility result for operator \eqref{4.17-ww-1}.
\begin{corollary}
\label{cor-C4.3}
Let a cut-off function $\chi\in X_+^4$.
Then the  operator $\mathbf M :\mathbb H  \to  \mathbb F $
is invertible.
\end{corollary}
Summarizing the above obtained results we can make the following conclusions.\\
 Consider LBDIE system \eqref{2.71-w}-\eqref{2.76-w} with arbitrary right-hand sides,
\begin{align}
\label{4.71-w}
&
\beta \,u_2 + { {\cal N}_\chi }\,u_2 -{  V_\chi}\,\psi_2 +{  W_\chi}\,\varphi_2 =h_1 \;\;\;\text{in}\;\;\; \Omega_2,\\
\label{4.72-w}
&
  {\cal N}_\chi^+ \,u_2 -{ \mathcal V_\chi}\psi_2 +[(\beta - \mu)\,I +{ \mathcal W_\chi}]\varphi_2 =h_2 \; \;\;\text{on}\;\; \;  S,\\
&
\label{4.73-w}
{ \mathcal K_\omega}  \psi_2 - { \mathcal M_\omega} \varphi_2 =h_3  \;\;\;\text{on}\;\;\;S,\\
&
\label{4.74-w}
\psi_2-\psi_1=h_4 \;\;\;\text{on}\;\;\;S,\\
&
\label{4.75-w}
\varphi_2-\varphi_1=h_5 \;\;\;\text{on}\;\;\;S,\\
&
\label{4.76-w}
u_1+ {  V_\omega}\,\psi_1-{  W_\omega}\,\varphi_1 = h_6 \;\;\;\; \text{in}\;\;\; \Omega_1.
\end{align}
Theorem \ref{T4.1} and Corollaries  \ref{cor-C4.2} and \ref{cor-C4.3} imply the following assertion.
\begin{corollary}\label{C4.4}

(i) LBDIE system \eqref{4.71-w}-\eqref{4.76-w} with arbitrary right hand side data
\begin{align}
\label{4.79-w}
(h_1,\cdots,h_6) \in \mathbf{Y}:=H^{1 }(\Omega_2)\times H^{\frac{1}{2}}(S)\times H^{-\frac{1}{2}}(S)\times H^{-\frac{1}{2}}(S)\times H^{\frac{1}{2}}(S)
\times   H^{1}_{comp}(\Omega_1)
\end{align}
is uniquely solvable in the space
\begin{align}
\label{4.78-w}
\mathbf{X}:=H^{1 }(\Omega_2)\times H^{-\frac{1}{2}}(S)\times H^{\frac{1}{2}}(S)\times H^{-\frac{1}{2}}(S)
\times
H^{\frac{1}{2}}(S)
\times \big( H^{1}_{loc}(\Omega_1) \cap Z(\Omega_1)\big).
\end{align}

(ii) LBDIE system \eqref{4.71-w}-\eqref{4.76-w} with arbitrary right hand side data
\[
(h_1,\cdots,h_6) \in \mathbf{F}:=H^{1, \,0}({\Omega_2};\Delta)\!\times \!H^{\frac{1}{2}}(S)\!\times \!
H^{-\frac{1}{2}}(S)\!\times\! H^{-\frac{1}{2}}(S)\!\times\! H^{\frac{1}{2}}(S)
\times   H^{1,\,0}_{comp}(\Omega_1; A_1)
\]
is uniquely solvable in the space $\mathbf{H}$ defined in \eqref{2.78-w},
\begin{align*}
\mathbf{H}=H^{1,0}({\Omega_2}; A_2)\times H^{-\frac{1}{2}}(S)\times H^{\frac{1}{2}}(S)\times H^{-\frac{1}{2}}(S)
\times H^{\frac{1}{2}}(S) \times \big( H^{1,0}_{loc}({\Omega_1}; A_1) \cap Z({\Omega_1})\big).
\end{align*}
In particular, under conditions \eqref{2.20-w1}, system \eqref{2.71-w}-\eqref{2.76-w} is uniquely solvable in the space
$\mathbf{H}$.

In both cases, (i) and (ii), the solution continuously depends on the right hand side data provided $ {\rm supp}\, h_6 \subset \overline{\Omega}_0$, where
$\overline{\Omega}_0$ is a fixed compact subset of $\Omega_1$.
\end{corollary}

Finally, Corollary \ref{C4.4}(ii), equivalence Theorem \ref{T2.3}, and uniqueness Theorem \ref{T2.4}
lead to the following assertion.
\begin{theorem}
\label{tDDD}
Let conditions \eqref{2.20-w1} hold.
Transmission problem \eqref{2.16-w}--\eqref{2.19-w} is uniquely solvable and the solution continuously depends on the right hand side data provided ${\rm supp}\, f_1 \subset \overline{\Omega}_0$, where
$\overline{\Omega}_0$ is a fixed compact subset of $\Omega_1$.
\end{theorem}

\section{Coupling of variational and non-local BIE approach}
\label{S3}

Here we present an alternative approach for investigation of transmission problem \eqref{2.16-w}-\eqref{2.20-w1}.
We apply the non-local approach and reformulate the transmission problem
in variational form. To this end, we  recall
the first Green identity \eqref{2.14-wG} in $\Omega_2$,
\begin{align}
\label{2.90-w}
\int\limits_{{\Omega_2} }  [a^{(2)}_{kj} \,\pa_{j}u_2 \, \overline{\pa_{k} {v}} -
  \omega^2\kappa_2  u_2  \overline{v }\,]dx-\langle { T^+_2}u_2 \,,\,\overline{\gamma^+v}\rangle_{S }
  =
- \int\limits_{{\Omega_2} } A_2 u_2 \  \overline{v }\,dx,\;\;
   \forall\ u_2\in H^{1,\,0}(\Omega_2; A_2),\ v\in H^{1}(\Omega_2).
\end{align}
Assuming that a pair  $ (u_2, u_1)\in H^{1,\,0}(\Omega_2; A_2)\times
\big( H^{1,0}_{loc}({\Omega_1}; A_1) \cap Z({\Omega_1})\big) $ solves transmission
problem \eqref{2.16-w}-\eqref{2.20-w1} and implementing the Steklov-Poincar\'{e}
type relation  \eqref{2.63-w}, we reduce \eqref{2.90-w} to equation
\begin{align}
&
\label{2.91-w}
\mathfrak{B}(u_2,v)=\mathfrak{F}(v)\quad
\forall\  v\in H^{1}(\Omega_2),
\end{align}
where  $\mathfrak{B}$ is a sesquilinear form  and $\mathfrak{F}$ is
an antilinear functional defined, respectively, as
\begin{align}
&
\label{2.92-w}
\mathfrak{B}(u_2,v):=\int\limits_{{\Omega_2} }  [a ^{(2)} _{kj}(x)\,\pa_{j}u_2 (x)\;
 \overline{\pa_{k} {v}(x)}-   \omega^2\kappa_2(x) u_2 (x)\, \overline{v(x)}\,]dx-
 \langle { \mathcal K_\omega^{-1} \,\mathcal M_\omega}(\gamma^+u_2) \,,\,\overline{\gamma^+v}\rangle_{S },\\
 &
\label{2.93-w}
\mathfrak{F}(v):=- \int\limits_{{\Omega_2}} f_2(x)\,\overline{v(x)}\,dx+\langle \Phi_\omega \,,\,\overline{\gamma^+v}\rangle_{S},
\end{align}
with
$
{ \Phi_\omega}:={\mathcal K_\omega} ^{-1} \,[{ \Psi_\omega}f_1 -{ \mathcal M_\omega}\varphi_0]+\psi_0\in H^{-\frac{1}{2}}(S).
$
Here the operators  ${\mathcal K_\omega}$, ${\mathcal M_\omega}$ and  ${\Psi_\omega}$
are defined by  relations \eqref{2.61-w}--\eqref{2.60-w}.

We associate with equation \eqref{2.91-w} the following variational problem (in a wider
space):\\
$\bullet$ \textsl{
 Find a function $u_2\in H^1({\Omega_2})$ satisfying \eqref{2.91-w}.}

Let us first prove the following equivalence theorem.
\begin{theorem}
\label{T3.1}
Let conditions \eqref{2.20-w1} be fulfilled.

(i) If a pair $(u_2,u_1)\in   H^{1,\,0}({\Omega_2};A_2)\times \big( H^{1,\,0}_{loc}({\Omega_1};A_1)\cap Z({\Omega_1})\big)$
solves transmission problem \eqref{2.16-w}--\eqref{2.20-w1},
then the function $u_2$  solves  variational equation  \eqref{2.91-w}.

(ii) Vice versa, if a function $u_2 \in H^1({\Omega_2})$ solves   variational equation  \eqref{2.91-w}, then the pair $(u_2,u_1)$,
where
\begin{eqnarray}
\label{2.96-w}
u_1(y)={ {\cal P}_\omega }  f_1 (y)- {  V_\omega}(T^+_2u_2-\psi_0)(y)+{  W_\omega}(\gamma^+u_2-\varphi_0)(y),
\;\;\;\; y\in {\Omega_1},
\end{eqnarray}
belongs to the class $H^{1,\,0}({\Omega_2};A_2)\times \big( H^{1,\,0}_{loc}({\Omega_1};A_1)\cap Z({\Omega_1})\big)$ and
solves  transmission problem \eqref{2.16-w}--\eqref{2.20-w1}.
\end{theorem}
\begin{proof} (i) The first part of the theorem follows from   the  derivation of  variational equation \eqref{2.91-w}.

(ii) To prove the second part we proceed as follows.
  If $u_2$   solves
\eqref{2.91-w}, then the equation particularly holds for
 $v\in \mathcal{D}(\Omega_2)$,  which implies that $u_2$
is a  solution of  equation \eqref{2.17-w} in the sense of distributions and
evidently $u_2\in H^{1,\,0} ({\Omega_2};A_2)$   since $f_2\in H^0(\Omega_2)$  in view of
\eqref{2.20-w1}. Therefore the canonical co-normal derivative $T^+_2u_2\in H^{-\frac{1}{2}}(S)$ is
well-defined in the sense of \eqref{2.14-w}.

Further, it is easy to see that   function \eqref{2.96-w} is well-defined, solves the differential equation \eqref{2.16-w}
due to \eqref{2.57-ww-1}--\eqref{2.57-ww-2}, and belongs to the space $H^{1,\,0}_{loc}({\Omega_1};A_1)\cap Z({\Omega_1})$
in view of \eqref{2.20-w1}. Therefore, the canonical co-normal derivative $T^-_1u_1\in H^{-\frac{1}{2}}(S)$ is
well-defined in the sense of \eqref{2.15-w1}
as well.

 In order to show that   transmission conditions \eqref{2.18-w}--\eqref{2.19-w} are also satisfied,
we write Green's identity \eqref{2.90-w} for $u_2$ and arbitrary $v\in H^1({\Omega_2})$ and subtract it from \eqref{2.91-w} to obtain:
\[
 \langle T^+_2 u_2- \mathcal K_\omega^{-1} \,\mathcal M_\omega(\gamma^+u_2) - \Phi_\omega\,,\,\overline{\gamma^+v}\rangle_S =0.
\]
Whence
  $   {  T^+_2}u_2- { \mathcal K_\omega^{-1} \,\mathcal M_\omega}(\gamma^+u_2) -{ \Phi_\omega}=0$ on $S$,
 i.e.,
$
   {  T^+_2}u_2-\psi_0- {\mathcal K}_\omega^{-1} \,{ \mathcal M_\omega}(\gamma^+u_2)
   -{\mathcal K}_\omega^{-1} \,({ \Psi_\omega}f_1-{ \mathcal M_\omega}\varphi_0)=0$ on $S$,
which is equivalent to the condition
\[
 { \mathcal K_\omega} ( {  T^+_2}u_2-\psi_0)-   \,{ \mathcal M_\omega}(\gamma^+u_2-\varphi_0)
   - { \Psi_\omega}f_1 =0 \;\;\;\text{on}\;\;\;S.
\]
In turn, in view of \eqref{2.61-w}--\eqref{2.60-w}, the latter implies
\begin{align}
\label{2.97-w}
  {  T_1^+\,w -i\,\alpha\gamma^+w }=0  \;\;\;\text{on}\;\;\;S,
\end{align}
where
\begin{align}
\label{2.98-w}
 w := {  V_\omega}({  T^+_2}u_2-\psi_0)-{ W_\omega}(\gamma^+u_2-\varphi_0) -{ {\cal P}_\omega }  f_1  \;\;\;\; \text{in}\;\;\;{\Omega_2}.
\end{align}
The function $w$ satisfies the homogeneous equation $A_1w =0$ in ${\Omega_2}$ in view of \eqref{2.98-w}
and the homogeneous Robin condition \eqref{2.97-w}.
As in the proof of Theorem \ref{T2.3},  we can deduce that
{$\gamma^+ w =0$ and $T^+_1 w =0$ on $S$ for real $\alpha\neq 0$}, implying $w=0$ in $\Omega_2$.
Therefore, for the function $u_1$ defined in \eqref{2.96-w} by Lemma \ref{L.C1} we have:
\[
\gamma^-u_1=\gamma^-u_1 +\gamma ^+w =\gamma^+u_2 -\varphi_0, \;\;\;\;\;
{  T^-_1}u_1={  T^-_1}u_1 + T_1^+w ={  T^+}u_2-\psi_0,
\]
which completes the proof.
\end{proof}

\begin{corollary}
\label{C3.2}
The homogeneous  variational  problem \eqref{2.91-w} (with $\mathfrak{F}=0$) possesses only the trivial solution.
\end{corollary}
\begin{proof}
It follows from the uniqueness and  equivalence Theorems \ref{T2.4} and \ref{T3.1}, respectively.
\end{proof}

 Further we analyse  the coercivity properties of the sesquilinear form $\mathfrak{B}$.
\begin{lemma}
\label{L3.3}
For the sesquilinear form $\mathfrak{B}$ defined in \eqref{2.92-w}   there  are real constants
$C^*_1>0$, $C^*_2>0$, and $C^*_3$ such that
\begin{align*}
\begin{array}{ll}
|\mathfrak{B}(u,v)|\leq C^*_1 \,\|u  \|_{_{H^1({\Omega_2})}}\, \|v  \|_{_{H^1({\Omega_2})}} & \;\;\;\forall\, u,\,v\in H^1({\Omega_2}),
\\
\mbox{\rm Re}\,\mathfrak{B}(u,u)\geq  C^*_2 \,\|u  \|^2_{_{H^1({\Omega_2})}} -
C^*_3\,\|u  \|^2_{_{H^0({\Omega_2})}} &\;\;\; \forall\, u \in H^1({\Omega_2}).
\end{array}
\end{align*}
\end{lemma}
\begin{proof} The first equality follows from \eqref{2.92-w} by the Cauchy-Schwartz inequality and the trace theorem.
To prove the second inequality, we use the positive definiteness of the matrix
 $\mathbf{a}_2=\big[\,a^{(2)}_{kj}\,\big]_{k,j=1}^3$,  Remark \ref{rC.4}, and  the trace theorem to obtain
 \begin{align*}
 \mbox{\rm Re}\,\mathfrak{B}(u,u) &\geq
 c_1\,\|u  \|^2_{_{H^1({\Omega_2})}} -c_2\,\|u\|^2_{_{H^0({\Omega_2})}} +
  C_1 \| \gamma^+ u \|^2_{H^{\frac{1}{2}}(S)}-C_2\| \gamma^+ u \|^2_{H^0(S)}
  \\
&
\geq c_1 \,\|u  \|^2_{_{H^1({\Omega_2})}} - c_2\,\|u\|^2_{_{H^0({\Omega_2})}}-
  C _2\,  \| \gamma^+ u  \|^2_{_{H^\delta({S})}}
\geq c_1 \,\|u  \|^2_{_{H^1({\Omega_2})}} - c_2\,\|u\|^2_{_{H^0({\Omega_2})}}-
c_3\,  \| u  \|^2_{_{H^{\ha+\delta}({\Omega_2})}},
 \end{align*}
 where $c_1>0$, $c_2=\omega ^2\, \max\limits_{\overline{\Omega}_2}\, \kappa_2(x)$,
 $C_1>0$ and $C_2\geqslant 0$ are the constants involved in \eqref{C-4-dd}, $c_3>0$,
  and $\delta $ is an arbitrarily small positive number.
 Now, by Ehrling's lemma,
  cf. e.g. Theorem 7.30 in Renardy et al\cite{Renardy-Rogers2004},
 for arbitrarily small positive number $\varepsilon $  there is a positive constant $C(\varepsilon)$, such that
 \[
 \| u  \| _{_{H^{\ha+\delta}({\Omega_2})}}\leq \varepsilon  \| u  \| _{_{H^{1}({\Omega_2})}} +
 C(\varepsilon) \| u  \| _{_{H^{0}({\Omega_2})}},
 \]
which  completes the proof.
\end{proof}

Now we prove the following existence results.
 \begin{theorem}
\label{T5.4}
Let $\mathfrak{F}$ be a bounded linear functional on $H^1({\Omega_2})$. Then variational equation \eqref{2.91-w} is uniquely solvable in the space $H^1({\Omega_2})$.
\end{theorem}
\begin{proof}
By Lemma \ref{L3.3} the sesquilinear functional
$\mathfrak{B}_\lambda(u,v):=\mathfrak{B}(u,v)+ \lambda\,\langle u,v\rangle_{ \Omega_2 }$ with $\lambda>|C^*_3|$
is positive and bounded below on the space $H^1(\Omega_2)\times H^1(\Omega_2)$.
Due to the  Lax-Milgram lemma,
 $\mathfrak{B}_\lambda$ defines an invertible linear operator
$\mathbf{T}_\lambda:H^1(\Omega_2)\to \widetilde{H}^{-1}(\Omega_2)$ for $\lambda>|C^*_3|$.
Therefore for arbitrary $\lambda$ the operator $\mathbf{T}_\lambda$ is Fredholm with zero index (see, e.g.
  Theorem 2.33 in  McLean\cite{McL}), since the sesquilinear form  $\lambda\,\langle u,v\rangle_{\Omega_2}$
  defines a compact   imbedding  operator $\lambda \,{I}: H^1(\Omega_2)\to \widetilde{H}^{-1}(\Omega_2)$,
  where $ {I}$ is the identity operator.
By Corollary \ref{C3.2}  the operator
  $\mathbf{T}_0$ defined by the sesquilinear form $\mathfrak{B}(u,v)=\mathfrak{B}_0(u,v)$ possesses the trivial null-space, and consequently is invertible, which completes the proof.
\end{proof}
  \begin{theorem}
\label{T5.4-D}
Let conditions \eqref{2.20-w1} be fulfilled. Then
transmission problem \eqref{2.16-w}--\eqref{2.20-w1} is uniquely solvable in the space
 $H^{1,\,0}({\Omega_2};A_2)\times \big( H^{1,\,0}_{loc}({\Omega_1};A_1)\cap Z({\Omega_1})\big)$.
\end{theorem}
\begin{proof} If conditions \eqref{2.20-w1} are satisfied then the linear functional $\mathfrak{F}$ given by \eqref{2.93-w} is bounded,
\begin{align*}
|\mathfrak{F}(v)|\leq C \,\|v  \|_{_{H^1({\Omega_2})}} & \;\;\;\forall\, v\in H^1({\Omega_2}),
\end{align*}
which follows from the  Cauchy-Schwartz inequality, trace theorem, and properties of the operators
${\mathcal K_\omega}$, ${\mathcal M_\omega}$ and  ${\Psi_\omega}$
  defined by  relations \eqref{2.61-w}--\eqref{2.60-w}.
  \\
Therefore by equivalence Theorem \ref{T3.1} and existence Theorem \ref{T5.4} along with uniqueness Theorem \ref{T2.4}
 we conclude that the
transmission problem \eqref{2.16-w}--\eqref{2.20-w1} is uniquely solvable.
   \end{proof}
\begin{remark}
\label{r5.6}
From the equivalence Theorem  \ref{T2.3}  and  existence Theorem \ref{T5.4-D} it follows that
the LBDIE  system \eqref{2.71-w}--\eqref{2.76-w} possesses a unique solution in the space $\mathbf{H}$
defined by \eqref{2.78-w}.
However, this does not imply  the results obtained in Section \ref{S-4} concerning neither
  the invertibility of the localized
boundary-domain matrix integral operator
generated by the left hand side expressions in \eqref{2.71-w}--\eqref{2.76-w}    nor
  the solvability in the space $\mathbf{X}$ of system \eqref{4.71-w}-\eqref{4.76-w} with arbitrary
right hand side functions from the space $\mathbf{Y}$ (see \eqref{4.79-w}-\eqref{4.78-w}).
The case is that Theorems  \ref{T2.3}  and   \ref{T5.4-D} yield unique solvability of
system \eqref{2.71-w}--\eqref{2.76-w} with only special form right hand-side functions
represented by volume and surface integrals (see the right hand side functions in \eqref{2.71-w}--\eqref{2.76-w}).
\end{remark}
\appendix
\section{Classes of cut-off functions}
\label{Appendix A}

 Here we  present some  classes of  localizing cut off functions   (for details see Chkadua et al\cite{CMN3}).
\begin{definition}
\label{dA.1}
We say $\chi \in X^k$ for integer $k\geq 0$  if
$\chi
(x)=\breve{\chi}(|x|)$, $\breve{\chi} \in W_1^k(0,\infty)$ and
$\varrho\breve{\chi}(\varrho) \in L_1(0,\infty)$.

We say $\chi \in X^k_+$ for integer $k\geq 1$  if $\chi \in X^k$,
${\chi} (0)=1$, and  $\sigma_\chi(\omega)>0$ for all $\omega\in\R$, where
 \[
    \sigma_\chi(\omega):=\left\{
    \begin{array}{l}
  \dst  \frac{\hat{\chi}_s(\omega)}{\omega}>0\;\;\text{for}\;\; \omega\in\R \setminus\{0\}\,,\\[3mm]
   \dst \int _{0}^{\infty} \varrho \breve{\chi}\,(\varrho )  \,d\varrho\;\; \text{for}\;\;\omega=0\,,
    \end{array}
    \right.
\qquad\quad
  \hat{\chi}_s(\omega):=\int_{0}^{\infty}
  \breve{\chi}\,(\varrho ) \;\sin(\varrho\,\omega) \,d\varrho\,.
 \]
\end{definition}
The following lemma provides an easily verifiable sufficient condition for
non-negative non-increasing functions to belong to the class $X^k_+$.
\begin{lemma}[Chkadua et al\cite{CMN3}, Lemma 3.2]
\label{lA.2}
Let $k\ge 1$. If $\chi \in X^k$, $\breve{\chi} (0)=1$,
$\breve{\chi}(\varrho)\geq 0$ for all $\varrho\in (0,\infty)$, and
$\breve{\chi}$ is a non-increasing function on $[0,+\infty)$, then
${\chi} \in X^k_+$.
\end{lemma}
Here are some  particular examples of cut off functions,
\begin{eqnarray*}
& \chi_{ 1k } (x)=
 \left\{
 \begin{array}{lll}
   \displaystyle \Big[\,
 1-\frac{|x| }{ \varepsilon
 }\,\Big]^{k} & \text{for} & |x|< \varepsilon
 ,\\
0 & \text{for} & |x|\geq \varepsilon,
 \end{array}
 \right.
\;\;
  & \chi_{ 2k } (x)=
 \left\{
 \begin{array}{lll}
   \displaystyle \Big[\,
 1-\frac{|x|^2}{ \varepsilon^2
 }\,\Big]^{k} & \text{for} & |x|< \varepsilon
 ,\\
0 & \text{for} & |x|\geq \varepsilon,
 \end{array}
 \right.
\;\;
  \chi_{3} (x)=
 \left\{
 \begin{array}{lll}
    \displaystyle \exp \Big[\,{\frac{|x|^2}{|x|^2-\varepsilon^2}}\,\Big] &
    \text{for} & |x|< \varepsilon
  ,\\
0 & \text{for} & |x|\geq \varepsilon.
 \end{array}
 \right.
  \end{eqnarray*}
Due to Lemma \ref{lA.2} we have  $\chi_{1k} \in X^k_+$,  $\chi_{2k} \in X^k_+\cap C^{k-1}(\mathbb{R}^3)$,
and $\chi_{3} \in X^{\infty}_+\cap C^\infty (\mathbb{R}^3)$.

 \section{Properties of localized potentials}
 \label{Appendix B}

 Here we collect some theorems describing mapping properties of the localized potentials \eqref{2.25-w}
 and \eqref{2.27-w},  and    the localized boundary operators generated by them
 \begin{eqnarray}
\label{B-d1}
{\cal V}_\chi \,g(y):=-\int _{S} {  P_\chi}(x-y)\, g(x)\,dS_x,
\quad
{\cal W} _\chi\,g(y):=-\int _{S}
\big[\,{ T_2}(x, \partial_x)\,{  P_\chi}(x-y)\,\big]\,
\, g(x)\,dS_x,  \;\;\;y\in S.
\end{eqnarray}
 Note that  ${\cal V}_\chi$ is a weakly singular  integral operator
(pseudodifferential operator of order $-1$), while ${\cal W}_\chi$ is a
 singular integral operator (pseudodifferential operator of order $0$).

  Remark that if $S\in C^\infty$ and  a  cut off function $\chi$ is infinitely differentiable,
 then the localized potentials and the corresponding boundary operators have
 the same mapping properties as the corresponding harmonic potentials (see, e.g.,  Miranda\cite{Mir},
 Hsiao \& Wendland\cite{HW}). However, for
cut off functions of finite smoothness the localized potential operators  possess quite different properties,
in particular, their smoothness is reduced  and the smoothness exponents depend on the
smoothness of a cut off function  $\chi$.
Properties of the localized  potentials needed in our analysis in the main text are presented below
 (detailed proofs can be found in Chkadua et al \cite{CMN3,CMN6}).

\begin{theorem}[Chkadua et al\cite{CMN3}, Theorems 5.6 and 5.10]
\label{tB.4}
 The following operators are continuous
\begin{align*}
 {\cal P}_\chi
 :\;& H^s({\Omega_2}) \to H^{s+2,s}({\Omega_2};\Delta),\qquad
 -\ha<s<\ha, \qquad \chi  \in X^{1}.
\\
{  V_\chi}
   :\;& H^{s-\frac{3}{2}}(S) \to
 H^{s}( \Omega_2),\quad \ha<s<k+\ha\ ,\quad
\text{if}\;\; \chi  \in X^k, \quad   k= 1, 2, ...
\\
 :\;&  H^{s-\frac{3}{2}}(S) \to
 H^{s,s-1}(\Omega_2;\Delta) ,
 \quad \ha< s<\tha,  \quad\text{if}\;\; \chi  \in X^2,
 \\
 { W_\chi}   :\;&   H^{s-\frac{1}{2}}(S)\to H^{s}({\Omega_2}) ,
 \quad \ha< s<k-\ha\ ,\quad
\text{if}\;\;\chi  \in X^k, \quad   k=   2, 3,...
 \\
 :\;&  H^{s-\frac{1}{2}}(S)\to
 H^{s,s-1}(\Omega_2;\Delta) ,
 \quad \ha< s<\tha,  \quad\text{if}\;\; \chi  \in X^3,
\end{align*}
where
$
H^{t,\,r}(\Omega_2; \Delta ):=\{ u\in H^{t}(\Omega_2)\,:\, \Delta u\in H^{r}(\Omega_2) \}
$.
\end{theorem}

\begin{theorem}[Chkadua et al\cite{CMN3}, Corollary 5.12 and  Theorem 5.13]
\label{tB.7}
Let $\chi \in X^{2}$, $\psi\in H^{-\frac{1}{2}}(S)$, and
 $\varphi\in H^{\frac{1}{2}}(S)$.
Then there hold the following
relations on $S$
\begin{eqnarray}
&
 \label{3.8j}
\gamma ^+{  V_\chi}\psi={ \V_\chi}\psi,
 \qquad
 \gamma^+ {  W_\chi}\varphi = -\mu\, \varphi +{ \W_\chi} \varphi
  \;\;\;\text{with}\;\;
\mu(y):=\frac{1}{2}\,a^{(2)}_{kj}(y)\,n_k(y)\,n_j(y)>0,\;\;y\in S.
\end{eqnarray}
\end{theorem}

\begin{theorem}[Chkadua et al\cite{CMN3}, Theorem 5.14]
\label{tB.6}
Let $-\ha < s < \ha$.
 The following  operators are continuous,
\begin{eqnarray}
\label{op1}
{\cal V}_\chi &:& H^{s -\ha}(S)\to
H^{s+\ha}(S), \quad  \chi  \in   X^1, \\
\nonumber
 {\cal W}_\chi  &:&  H^{s+\ha }(S)\to
 H^{s+\ha }(S), \quad  \chi  \in   X^2.
\end{eqnarray}
Moreover, operator \eqref{op1} is Fredholm with zero index.
\end{theorem}

\begin{remark}
\label{rB.8}
The principal homogeneous symbols of the boundary pseudodifferential operators ${\cal V}_\chi$,
 $-\mu \, {I}+{\cal W}_\chi$, and $(\beta-\mu) \, {I}+{\cal W}_\chi$,
calculated in a local coordinate system with the origin at a point ${\widetilde{y}}\in S$ and the third axis coinciding with the normal
vector at the point ${\widetilde{y}}\in S$, read as
\begin{align}
&
\label{B-d-1}
{\mathfrak{S}}_0({ \cal V}_\chi;{\widetilde{y}}, \xi')=\frac{1}{2\,|\xi'|},
\qquad
 {\mathfrak{S}}_0\big(-\mu \, {I}+{ \cal W}_\chi;{\widetilde{y}}, \xi'\big)
 =-\frac{1}{2}\, \widetilde{a}^{(2)}_{33}({\widetilde{y}}) -
 \frac{i}{2}  \sum\limits_{l=1}^{2}\widetilde{a}^{(2)}_{3l}({\widetilde{y}})\,\frac{\xi_l}{|\xi'|}\,,
 \\
 \label{B-d-2-1}
 &
{\mathfrak{S}}_0\big((\beta -  \mu )\,{I}+{ \mathcal W_\chi};{\widetilde{y}},\xi'\big)=
  \frac{1}{2}\,\Big[
  2\widetilde{\beta}-\widetilde{a}^{(2)}_{33}({\widetilde{y}}) -i\sum\limits_{l=1}^{2}
  \widetilde{a}^{(2)} _{3l}({\widetilde{y}})\,\frac{\xi_l}{|\xi'|}
 \Big]\,,\quad \xi'\in \mathbb{R}^2\setminus \{0\},
\end{align}
where $\big[ \widetilde{a}^{(2)}_{kj}({\widetilde{y}})\,\big]_{k,j=1}^3= \big[ \lambda_{pk}({\widetilde{y}})\,  a^{(2)}_{pq}({\widetilde{y}})\,\lambda_{qj}({\widetilde{y}})\,\big]_{k,j=1}^3= \Lambda({\widetilde{y}})^\top\, \mathbf{a}_2({\widetilde{y}})\, \Lambda({\widetilde{y}})  $ is a positive definite matrix and
\[
 \widetilde{{\beta}} ({\widetilde{y}})=\frac{1}{3}\big [\, \widetilde{a}^{(2)} _{11}({\widetilde{y}})+
 \widetilde{a}^{(2)} _{22}({\widetilde{y}})+ \widetilde{a}^{(2)}_{33}({\widetilde{y}})\, \big]>0\,.
\]
Here $\mathbf{a}_2({\widetilde{y}})=[a^{(2)}_{kj}({\widetilde{y}})]_{k,j=1}^3$
and  $\Lambda({\widetilde{y}})=[\lambda_{kj}({\widetilde{y}})]_{3\times 3}$ is
an orthogonal matrix with the property $\Lambda^\top \, n({\widetilde{y}})=(0,0,-1)^\top$,
where $n({\widetilde{y}})$ is the outward unit normal vector at the point ${\widetilde{y}}\in S$.
Therefore  $\lambda_{p3}({\widetilde{y}})=-n_p({\widetilde{y}}),$ $p=1,2,3.$ In view of \eqref{3.8j}
it is evident that
$
\frac{1}{2}\, \widetilde{a} ^{(2)}_{33}({\widetilde{y}})=\frac{1}{2}\,
\lambda_{p3}({\widetilde{y}})\, a^{(2)} _{pq}({\widetilde{y}})\,\lambda_{q3}({\widetilde{y}})=\mu({\widetilde{y}})>0\,.
$
\end{remark}

 \section{Properties of radiating potentials}
 \label{Appendix C}
  \setcounter{equation}{0}

The layer potentials defined by \eqref{2.37-w}
and the volume potential (cf. \eqref{2.39-w})
\[
{\bf P}_\omega  \,f(y):=\int_{\mathbb{R}^3} \Gamma(x-y,\omega ) \,f(x)\,dx,\quad y\in {\mathbb{R}}^3,
\]
 have the following properties (for details see  Jentsch et al\cite{JN2}).
\begin{lemma}
\label{L.C1}
\mbox{\rm (i)} The following operators are continuous
\[
{\begin{array}{ccll}
{  V_\omega}  &\;:\;&  H^{-\frac{1}{2}}(S)\to H^{1}({{\Omega_2}}, A_1)
\qquad \big[\, H^{-\frac{1}{2}}(S)\to H^{1}_{loc}(\Omega_1, A_1)
                                       \cap Z({\Omega_1})\,\big],\\
{  W_\omega}   &\;:\;&  H^{\frac{1}{2}}(S)\to H^{1}({{\Omega_2}}, A_1)
\;\;\,\qquad\big[\, H^{\frac{1}{2}}(S)\to H^{1}_{loc}(\Omega_1, A_1)
                                       \cap Z({\Omega_1})\,\big],\\
{ {\bf P}_\omega } &\;:\;& H^0_{comp}(\mathbb{R}^3) \to  H^2_{loc}(\mathbb{R}^3)\cap Z(\mathbb{R}^3).
\end{array}}
\]
Moreover,
\[
A_1 { {\bf P}_\omega }f=f \;\;\text{in}\;\;\mathbb{R}^3\;\;\text{for}\;\;f\in H^0_{comp}(\mathbb{R}^3).
\]
\mbox{\rm (ii)} For $h\in H^{-\frac{1}{2}}(S)$ and $g\in H^{\frac{1}{2}}(S)$
the following jump relations  hold true
\begin{align}
&
\gamma^+ { V_\omega}\,h =\gamma^- {  V_\omega}(h)= { \mathcal V_\omega}(h),
\quad
T_1 ^\pm {  V_\omega}\,h =\Big(\pm\, \frac{1}{2}I+{ \mathcal W_\omega'}\Big)h \;\;\;\text{on}\;\;\; S,
\nonumber
\\
&
\gamma^\pm  W_\omega  \,g  =\Big( \mp\frac{1}{2}I+{ \mathcal W_\omega}\Big)g,
\quad
\label{2.45-w}
T_1^+{ W_\omega}\,g={  T_1^-}{  W_\omega}\,g=:{ \mathcal L_\omega}\,g   \;\;\;\text{on}\;\;\; S,
\end{align}
where  $I$ stands for the identity operator, and
\begin{align}
\label{2.46-w}
&
{ \mathcal V_\omega}\,h(y):=-\int_S\Gamma(x-y,\omega)\,h(x)\,dS_x,\quad y\in S,\\
&
\label{2.47-w}
{ \mathcal W_\omega}\,g(y):=-\int_S [{ T_1}(x,\pa_x)\Gamma(x-y,\omega))]\,g(x)\,dS_x, \quad y\in S,\\
 &
\label{2.48-w}
{ \mathcal W_\omega'}\,h(y):=-\int_S [{ T_1}(y,\pa_y)\Gamma(x-y,\omega))]\,h(x)\,dS_x, \quad y\in S,
\end{align}
$\Gamma(x,\omega)$ is the radiating fundamental solution defined by \eqref{2.40-w}.

\mbox{\rm (iii)} The following operators are continuous,
\begin{align*}
&
{ \mathcal V_\omega}: H^{-\frac{1}{2}}(S)\to H^{\frac{1}{2}}(S) , \quad
{ \mathcal W_\omega}: H^{\frac{1}{2}}(S)\to H^{\frac{1}{2}}(S), \quad
{ \mathcal W_\omega'}: H^{-\frac{1}{2}}(S)\to H^{-\frac{1}{2}}(S),  \quad
{ \mathcal L_\omega} : H^{\frac{1}{2}}(S)\to H^{-\frac{1}{2}}(S).
\end{align*}

\mbox{\rm (iv)} The operators ${ \mathcal W_\omega}$, and ${ \mathcal W_\omega'}$
  are compact, since they have weakly singular kernel-functions
of   the  type $O(|x-y|^{-1})$, ${ \mathcal V_\omega}$ is a pseudodifferential operator of order $-1$
with positive principal homogeneous symbol, ${\mathfrak{S}}_0\big({ \mathcal V_\omega}; y, \xi'\big)>0$,
and ${ \mathcal L_\omega} $ is a singular integro-differential operator
(pseudodifferential operator of order $1$) with negative principal homogeneous symbol,
${\mathfrak{S}}_0\big({ \mathcal L_\omega}; y, \xi'\big)<0$; moreover
$$
{\mathfrak{S}}_0\big({ \mathcal L_\omega}; y, \xi'\big)
 =-\big[{4 {\mathfrak{S}}_0\big({ \mathcal V_\omega};y, \xi'\big)}\big]^{-1} <0,
 \quad \xi'\in \mathbb{R}^2\setminus \{0\}, \;\;\;y\in S.
$$
\end{lemma}

\begin{lemma}
\label{L.C2}
Let ${\mathcal K_\omega}$ and ${\mathcal M_\omega}$ be defined by \eqref{2.61-w} and  \eqref{2.62-w} with
$\alpha>0$.   The  following operators are invertible
\[
{ \mathcal K_\omega} : H^{-\frac{1}{2}}(S)\to H^{-\frac{1}{2}}(S) , \qquad
{ \mathcal M_\omega} :  H^{\frac{1}{2}}(S)\to H^{-\frac{1}{2}}(S) .
\]

\end{lemma}

\begin{remark}
\label{rC.3}
The principal homogeneous symbols of the pseudodifferential operators
$\mathcal K_\omega$, $\mathcal M_\omega$, $\mathcal M_\omega^{-1}\mathcal K_\omega$, and
${\mathcal K}_\omega^{-1}{\mathcal M_\omega}$ and
calculated in a local coordinate system described in Remark \ref{rB.8} satisfy the relations
\begin{align}
&
\label{C-2}
{\mathfrak{S}}_0\big({ { \mathcal K_\omega}};y , \xi'\big)= {1}/{2},
\qquad
 {\mathfrak{S}}_0\big({ \mathcal M_\omega}; y, \xi'\big)
 =- \big[{4 \,{\mathfrak{S}}_0\big({ \mathcal V_\omega};y, \xi'\big)}\big]^{-1}  < 0,
 \\
   &
  \label{C-3}
 {\mathfrak{S}}_0\big({ \mathcal M}_\omega ^{-1}{\mathcal K_\omega};y,\xi'\big)
 =-2\,{\mathfrak{S}}_0\big({ \mathcal V_\omega};y, \xi'\big) < 0 ,
 \\
 &
 \label{C-3-1}
 {\mathfrak{S}}_0\big({ \mathcal K}_\omega ^{-1}{\mathcal M_\omega};y , \xi'\big)
 =-\big[{2\,{\mathfrak{S}}_0\big({ \mathcal V_\omega};y , \xi'\big)}\big]^{-1}<0,
 \quad \xi'\in \mathbb{R}^2\setminus \{0\}, \;\;\;y\in S.
\end{align}
\end{remark}

\begin{remark}
\label{rC.4}
 The principal homogenous symbols  of the operators
 ${ \mathcal V_\omega}$, ${ \mathcal M_\omega}$,  and  $-{ \mathcal K_\omega}^{-1}{ \mathcal M_\omega}$ are positive
  in view of Lemma $\mbox{\rm\ref{L.C1} (iv)}$ and Remark \ref{rC.3}. Therefore
  it can be shown   that there are constants $C_1>0$ and $C_2\geqslant 0$ such that
  the following inequalities hold (cf, e.g.,   Theorem 6.2.7 in Hsiao \& Wendland\cite{HW})
\begin{align}
  &
\langle {\psi, \mathcal{V}}_\omega \psi\rangle_S\geq C_1 \| \psi\|^2_{H^{-\frac{1}{2}}(S)}-C_2\| \psi\|^2_{H^{-\frac{3}{2}}(S)}
\quad \forall\, \psi\in H^{-\frac{1}{2}}(S),
 \nonumber
\\
&
\langle   \mathcal{M}_\omega\psi\,,\,\psi\rangle_S
\geq C_1 \| \psi\|^2_{H^{\frac{1}{2}}(S)}-C_2\| \psi\|^2_{H^{-\ha}(S)}
\quad \forall\, \psi\in H^{\frac{1}{2}}(S),
 \nonumber
\\
\label{C-4-dd}
&
\langle { -\mathcal{K}}_\omega^{-1} \mathcal{M}_\omega\psi\,,\,\psi\rangle_S
\geq C_1 \| \psi\|^2_{H^{\frac{1}{2}}(S)}-C_2\| \psi\|^2_{H^{-\ha}(S)}
\quad \forall\, \psi\in H^{\frac{1}{2}}(S).
\end{align}
\end{remark}

\noindent
\ack{\textit{\bf Acknowledgement}}\\
 {This research was supported by the grants EP/H020497/1: ``Mathematical
Analysis of Localized Boundary-Domain Integral Equations for
Variable-Coefficient Boundary Value Problems'' and EP/M013545/1: ``Mathematical
Analysis of Boundary-Domain Integral Equations for Nonlinear PDEs'', from
the EPSRC, UK.}



\end{document}